\documentclass[12pt]{article}
\usepackage{amssymb,amsthm,hyperref,amsmath,bm,amsfonts}
\usepackage{arydshln}
\usepackage{verbatim}
\usepackage{graphicx}
\usepackage{float}

\usepackage{subfigure}
\usepackage{multirow}
\usepackage{subfigure}
\usepackage{color}
\usepackage{appendix}

\newtheorem{theorem}{Theorem}[section]
\newtheorem{lemma}[theorem]{Lemma}
\newtheorem{prop}{Proposition}[section]
\newtheorem{assumption}{Assumption}[section]
\newtheorem{remark}{Remark}[section]
\numberwithin{equation}{section}

\title{Boundary Conditions for Hyperbolic Relaxation
Systems with Characteristic Boundaries of Type II}
\author{
Yizhou Zhou\thanks{E-mail: zhouyz16@mails.tsinghua.edu.cn}\\
\small{\textit{Department of Mathematical Sciences}},\\
\small{\textit{Tsinghua University, Beijing 100084, China.}}\\
Wen-An Yong\thanks{E-mail: wayong@tsinghua.edu.cn}\\
{\small{\textit{Department of Mathematical Sciences}}},\\
{\small{\textit{Tsinghua University, Beijing 100084, China.}}} \\
}

\date{\today}
\begin{document}
\maketitle{}

\begin{abstract}
This paper is a continuation of our preceding work on hyperbolic relaxation systems with characteristic boundaries of type I. 
Here we focus on the characteristic boundaries of type II, where the boundary is characteristic for the equilibrium system and is non-characteristic for the relaxation system. 
For this kind of characteristic initial-boundary-value problems (IBVPs), we introduce a three-scale asymptotic expansion to analyze the boundary-layer behaviors of the general multi-dimensional linear relaxation systems. 
Moreover, we derive the reduced boundary condition under the Generalized Kreiss Condition by resorting to some subtle matrix transformations and the perturbation theory of linear operators.
The reduced boundary condition is proved to satisfy the Uniformed Kreiss Condition for characteristic IBVPs. 
Its validity is shown through an error estimate involving the Fourier-Laplace transformation and an energy method based on the structural stability condition. 
\end{abstract}

\hspace{-0.5cm}\textbf{Keywords:}
\small{hyperbolic relaxation problems, boundary conditions, characteristic boundaries,
 Kreiss condition, energy estimate, Laplace transform}\\

\hspace{-0.5cm}\textbf{AMS subject classification:} \small{35L50, 76N20}

\newpage
\section{Introduction}\label{section1}

This paper is a continuation of our preceding work \cite{ZY} on hyperbolic relaxation systems with characteristic boundaries of type I. Now we consider other characteristic boundary conditions for multi-dimensional linear hyperbolic relaxation systems
\begin{equation}\label{1.1}
  U_t+\sum_{j=1}^{d}A_j U_{x_j}=\frac{QU}{\epsilon}.
\end{equation}
Here $U\in \mathbb{R}^n$ is unknown, $A_j$ ($j=1,2,...,d$) and $Q$ are given $n\times n$ constant matrices, and $\epsilon$ is a small positive parameter called relaxation time. 
Such equations can be understood as a linearization of certain partial differential equations which model a large number of irreversible phenomena. 
They arise in the kinetic theory \cite{Ga,Le,CFL,Mi}, thermal non-equilibrium flows \cite{Ze}, chemically reactive flows \cite{GM}, compressible viscoelastic flows \cite{Y3,CS}, non-equilibrium thermodynamics \cite{JCL,Mu,ZHYY}, traffic flows \cite{BK1,BK2}, and so on. 
We are interested in the limit as $\epsilon$ goes to zero---the so-called zero relaxation limit.


As in \cite{ZY}, we assume that the system \eqref{1.1} satisfies the structural stability conditions proposed in \cite{Y1}. It was shown in \cite{Y3} that the stability condition is quite reasonable and is respected by almost all relaxation models. Under this condition, it was proved in \cite{Y1} that, as $\epsilon$ goes to zero, the solution $U^\epsilon$ of the initial-value problem converges to that of the corresponding equilibrium system. Here the equilibrium system is a hyperbolic system of first-order partial differential equations and governs the limit (See Assumption \ref{Asp2.1} in Section \ref{section2}). 
For initial-boundary-value problems (IBVPs) of \eqref{1.1} which is defined only for the half spatial domain $x_1>0$, proper boundary conditions should be prescribed on the boundary $x_1=0$. Classically, it is important to distinguish that the boundary is characteristic or not \cite{MO, Kr, Y4}. 
Here, the boundary $x_1=0$ is said to be non-characteristic for the hyperbolic system \eqref{1.1} if the coefficient matrix $A_1$ is invertible. Otherwise, it is said to be characteristic.

In \cite{Y2}, the second author systematically studied the formulation of boundary conditions for the relaxation problems under the assumption that the boundary is non-characteristic for both the relaxation system and its equilibrium system. Particularly, he proposed the Generalized Kreiss Condition (GKC) which is essentially necessary to have a well-behaved relaxation limit, derived the so-called reduced boundary condition (satisfied by the limit), and showed the well-posedness of the reduced boundary condition together with the corresponding equilibrium system. 
In \cite{ZY}, we considered the characteristic boundaries of type I, where the boundary is non-characteristic for the equilibrium system and is characteristic for the relaxation system. For this kind of characteristic IBVPs, we modified the GKC, derived the reduced boundary condition, and verified the validity of the reduced boundary condition.


In the present work, we focus on the characteristic boundaries of type II, where the boundary is characteristic for the equilibrium system and is non-characteristic for the relaxation system. 
For this kind of characteristic IBVPs, the original GKC in \cite{Y2} is still valid. However, the arguments in \cite{Y2} have to be re-examined since the boundary is characteristic for the equilibrium system. In particular, we are motivated by \cite{XX} for a specific system and introduce a three-scale ($x_1,x_1/\sqrt{\epsilon},x_1/\epsilon$) asymptotic expansion to analyze the boundary-layer behaviors of the general multi-dimensional linear relaxation systems. Moreover, we derive the reduced boundary condition under the GKC by resorting to some subtle matrix transformations and the perturbation theory of linear operators \cite{Ka}.
To show the well-posedness, the reduced boundary condition is proved to satisfy the Uniformed Kreiss Condition proposed in \cite{MO} for characteristic IBVPs.
At last, the validity of the reduced boundary condition is shown through an error estimate involving the Fourier-Laplace transformation \cite{GKO} and an energy method based on the structural stability condition.


Here are three interesting corollaries of our results. 
(1) The three scales in the above expansion degenerate to two scales, namely there is no boundary-layer with length $\epsilon$, if the number of positive eigenvalues of $A_1$ equals to the number of the non-negative eigenvalues of the corresponding coefficient matrix for the equilibrium system. A specific example can be found in \cite{ZWF-Yong}.
(2) Surprisingly, our reduced boundary condition does not involve the characteristic modes corresponding to the zero eigenvalue. This property was shown in \cite{MO} to be crucial for the well-posedness of the characteristic IBVPs. 
(3) The multiplicity of the zero eigenvalue of the corresponding coefficient matrix for the equilibrium system is not larger than the number of non-equilibrium variables ($=$ the rank of the matrix $Q$). This can also be deduced from the well-known sub-characteristic condition for hyperbolic relaxation systems \cite{CLL,Y7}. 

The literature for IBVPs of the general relaxation systems seems quite rare. Except for \cite{Y2,ZY}, other works seem all for specific models. 
In \cite{XX}, 
by using a three-scale asymptotic expansion and the Laplace transform, the authors analyzed 
the boundary-layer behaviors and the asymptotic convergence for 
a specific linear model (Jin-Xin model) with constant coefficients. 
The similar argument was applied in \cite{Xu2} to a linearized multi-dimensional model proposed in \cite{KT}. 
In \cite{CH}, the authors studied a specific IBVP for the one-dimensional Kerr-Debye model in nonlinear optics. They justified the zero relaxation limit by exploiting the energy method and the entropy structure of the model. 
Recently, a nonlinear discrete-velocity model for traffic flows was investigated in \cite{BK1} and the reduced boundary condition was obtained by solving a boundary Riemann problem. 
In \cite{ZWF-Yong}, the authors studied a simple moment closure system with characteristic boundaries. 
Unlike these works, the present paper is devoted to studying the IBVPs for general multi-dimensional linear relaxation systems. 
The interested reader is referred to \cite{WX,XX2,Xu1,ZY2} for further works in this direction.

This paper is organized as follows. Some notations and basic assumptions are listed in Section \ref{section2}. In Section \ref{section3}, we present a three-scale formal asymptotic expansion for the IBVPs of multi-dimensional linear relaxation system with characteristic boundaries of type II. The reduced boundary condition is derived in Section \ref{section4}. In Section \ref{section5}, we prove the validity of the reduced
boundary condition. Some details are given in the Appendix.

\section{Preliminaries}\label{section2}

Consider initial-boundary-value problems (IBVPs) of multi-dimensional linear hyperbolic systems
\begin{eqnarray}\label{2.1}
\left\{{\begin{array}{*{20}l}
  \vspace{2mm}U_t+\sum\limits_{j=1}^d A_jU_{x_j}= QU/\epsilon,\qquad x_1>0,\ \hat{x}=(x_2,...,x_d) \in \mathbb{R}^{d-1},\ t>0,\\[3mm]
  BU(t,0,\hat{x})=b(t,\hat{x})
\end{array}}\right.
\end{eqnarray}
with $B$ a constant matrix. Throughout this paper, the coefficient matrix $A_1$ is assumed to be invertible. That is, the boundary $x_1=0$ is non-characteristic for the system \eqref{2.1}. 
Denote
$$
n^+ = ~\text{the number of positive eigenvalues of} ~A_1.
$$ 
By the classical theory of IBVPs for hyperbolic systems \cite{BS,GKO}, the boundary matrix $B$ should be a full-rank $n^+\times n$-matrix. 
As to the system \eqref{2.1}, we postulate that it satisfies the structural stability condition proposed in \cite{Y1} and the boundary condition satisfies the Generalized Kreiss Condition (GKC) proposed in \cite{Y2}. For the convenience of readers, we recall these conditions as follows. 


Firstly, we recall the structural stability condition only for our linear system \eqref{2.1}, while it was proposed in \cite{Y1} for nonlinear problems: 

(i) There is an invertible $n\times n$ matrix $P$ and an invertible $r\times r$ matrix $S$ 
such that 
$$
PQ=\left({\begin{array}{*{20}c}
  \vspace{1.5mm}0 & 0 \\
  0 & S
  \end{array}}\right)P.
$$

(ii) There exists a symmetric positive definite matrix $A_0$, called the symmetrizer, such that each $A_0A_j$ ($j=1,2,...,d$) is symmetric. 

(iii) The hyperbolic part and the source term are coupled in the following sense
$$
A_0Q+Q^*A_0\leq -P^*\left({\begin{array}{*{20}c}
  \vspace{1.5mm}0 & 0 \\
  0 & I_r
  \end{array}}\right)P.
$$
Here and below the superscript $*$ means transpose of matrices or vectors.

On this condition, we comment as follows. Under Condition (i), we can introduce $V=PU$ to obtain a system for $V$. Thus we may as well assume that $P=I_n$ and $Q=\text{diag}(0_{n-r}, S)$ in \eqref{2.1}. 
Moreover, Condition (iii) implies that $S$ is an $r\times r$ stable matrix and
$A_0$ is a block-diagonal matrix of the form $A_0=\text{diag}(A_{01}, A_{02})$ with the partition of $Q$ (see Theorem 2.2 in \cite{Y1}).  
Furthermore, we can assume that the matrices $A_j\ (j=1,2,...,d)$ are symmetric and the symmetrizer $A_0$ equals to $I_n$. 
In fact, with $\tilde{U} = A_0^{1/2}U$, the system \eqref{2.1} can be written as
\begin{equation}\label{2.2}
	\tilde{U}_t + \sum_{j=1}^d\tilde{A}_j\tilde{U}_{x_j} = \frac{1}{\epsilon}\tilde{Q}\tilde{U}, 
\end{equation}
where $\tilde{A}_j=A_0^{-1/2}(A_0A_j)A_0^{-1/2}$ is symmetric and $\tilde{Q}= A_0^{-1/2}(A_0Q)A_0^{-1/2}=\text{diag}(0, \tilde{S})$ with $\tilde{S}=A_{02}^{-1/2}(A_{02}S)A_{02}^{-1/2}$ negative definite. 
This $\tilde{Q}$ is also symmetric under the further assumption that 
\begin{equation}\label{2.3}
  A_0 Q =Q^* A_0, 
\end{equation} 
which is quite reasonable and is respected by almost all relaxation models as shown in \cite{Y3}.
In summary, we assume throughout this paper that the matrices $A_j$ in \eqref{2.1} are symmetric, $A_0$ is the unit matrix $I_n$, and $Q=\text{diag}(0,S)$ with $S$ a symmetric negative definite matrix. Correspondingly, we write  
$$
A_1 = \left({\begin{array}{*{20}c}
  \vspace{1.5mm}A_{11} & A_{12} \\
  A_{12}^* & A_{22}
  \end{array}}\right), \quad A_j = \left({\begin{array}{*{20}c}
  \vspace{1.5mm}A_{j11} & A_{j12} \\
  A_{j12}^* & A_{j22}
  \end{array}}\right),\quad
U=\left({\begin{array}{*{20}c}
  \vspace{1.5mm}u \\
                v
  \end{array}}\right), \quad B=(B_u,B_v).
$$

About the GKC, we define
$$
M=M(\xi,\omega,\eta) = A_1^{-1}\left(\eta Q - \xi I_n - i \sum_{j=2}^d \omega_j A_j\right).
$$
Here $\eta$ is a nonnegative real number, $\xi $ is a complex number satisfying $Re\xi > 0$, and $\omega=(\omega_2,...,\omega_d)\in \mathbb{R}^{d-1}$. Let $R_M^S(\xi,\omega,\eta)$ denote the right-stable matrix for $M=M(\xi,\omega,\eta)$ \cite{Y2}. By Lemma 2.3 in \cite{Y2}, the matrix $M(\xi,\omega,\eta)$ has $n^+$ stable eigenvalues and thereby $R_M^S(\xi,\omega,\eta)$ is an $n\times n^+$-matrix.
With this notation, the GKC can be stated as\\

\textbf{Generalized Kreiss Condition:} There exists a constant $c_K>0$ such that 
\begin{align*}
|\det\{BR_M^S(\xi,\omega,\eta)\}|\geq c_K \sqrt{\det\{R_M^{S*}(\xi,\omega,\eta)R_M^S(\xi,\omega,\eta)\}}
\end{align*}
for all $\eta\geq 0$, $\omega\in \mathbb{R}^{d-1}$ and $\xi$ with $Re\xi>0$.

At this point, we remark that the new problem \eqref{2.2} with boundary condition 
\begin{equation*} 
	b(\hat{x},t)=\left[BA_0^{-1/2}\right]\left[A_0^{1/2}U(0,\hat{x},t)\right]\equiv\tilde{B}\tilde{U}(0,\hat{x},t)
\end{equation*}
satisfies the GKC if so does the original problem \eqref{2.1}. Indeed, for \eqref{2.2} we have
\begin{align*}
\tilde{M}(\xi,\omega,\eta) :=& \tilde{A}_1^{-1}\left(\eta \tilde{Q}-\xi I_n - i\sum_{j=2}^d  \omega_j \tilde{A}_j\right) \\ 
=& A_0^{1/2}A_1^{-1}A_0^{-1/2}\left(\eta A_0^{1/2}QA_0^{-1/2}-\xi I_n-i\sum_{j=2}^d  \omega_j A_0^{1/2}A_j A_0^{-1/2}\right)\\
=& A_0^{1/2}A_1^{-1}\left(\eta Q-\xi I_n- i\sum_{j=2}^d  \omega_j A_j\right)A_0^{-1/2}.
\end{align*}
Then it is clear that $\tilde{R}_M^S(\xi,\omega,\eta)=A_0^{1/2}R_M^S(\xi,\omega,\eta)$ is a right-stable matrix for $\tilde{M}(\xi,\omega,\eta)$. Since the problem \eqref{2.1} satisfies the GKC, we deduce from Lemma 3.3 in \cite{Y2} that
\begin{align*}
|\det\{\tilde{B}\tilde{R}_M^S(\xi,\omega,\eta)\}|=&|\det\{BR_M^S(\xi,\omega,\eta)\}| \geq c_K \sqrt{\det\{R_M^{S*}(\xi,\omega,\eta)R_M^S(\xi,\omega,\eta)\}} \\[2mm]
=& c_K \sqrt{\det\{\tilde{R}_M^{S*}(\xi,\omega,\eta)A_0^{-1}\tilde{R}_M^S(\xi,\omega,\eta)\}}\\[2mm]
\geq & c_K c_A \sqrt{\det\{\tilde{R}_M^{S*}(\xi,\omega,\eta)\tilde{R}_M^S(\xi,\omega,\eta)\}}
\end{align*}
with $c_A$ a positive constant. This is exactly the GKC for the new problem.

At last, we make the following assumption 
\begin{assumption}\label{Asp2.1}
The boundary $x_1=0$ is characteristic for the corresponding equilibrium system 
$$
u_t+A_{11}u_{x_1}+\sum_{j=2}^dA_{j11}u_{x_j}=0.
$$
Namely, the matrix $A_{11}$ has zero eigenvalues while the matrix $A_1$ is invertible.
\end{assumption}
\noindent Note that $A_{11}$ and $A_{j11}$ are all symmetric under the structural stability condition. With this assumption, we denote 
\begin{align*}
n_1^+ =& ~\text{the number of positive eigenvalues of} ~A_{11},\\[2mm]
n_1^o =& ~\text{the number of zero eigenvalues of} ~A_{11}.
\end{align*}

In summary, our assumptions for the IBVP \eqref{2.1} are the structural
stability condition, the additional assumption \eqref{2.3}, the GKC, and Assumption \ref{Asp2.1}.

\section{Asymptotic Expansion}\label{section3}
To see the behavior of system \eqref{2.1} in the zero relaxation limit, we formally expand its solution in this section. 
Following the three-scale asymptotic expansion in \cite{XX}, we introduce 
$$
y=\frac{x_1}{\epsilon}, \quad z=\frac{x_1}{\sqrt{\epsilon}}
$$ 
and consider an approximate solution of the form
\begin{align}
U_\epsilon(t,x_1,\hat{x})=&
\left({\begin{array}{*{20}c}
  \vspace{1.5mm}\bar{u} \\
  \bar{v}
\end{array}}\right)(t,x_1,\hat{x})
+\left({\begin{array}{*{20}c}
  \vspace{1.5mm}\mu \\
  \nu
\end{array}}\right)(t,y,z,\hat{x};\epsilon)\nonumber\\[2mm]
\equiv&\left({\begin{array}{*{20}c}
  \vspace{1.5mm}\bar{u} \\
  \bar{v}
\end{array}}\right)(t,x_1,\hat{x})+
\left({\begin{array}{*{20}c}
  \vspace{1.5mm}\mu_0 \\
  \nu_0
\end{array}}\right)(t,y,\hat{x})+
\left({\begin{array}{*{20}c}
  \vspace{1.5mm}\mu_1 \\
  \nu_1
\end{array}}\right)(t,z,\hat{x})+\sqrt{\epsilon}
\left({\begin{array}{*{20}c}
  \vspace{1.5mm}\mu_2 \\
  \nu_2
\end{array}}\right)(t,z,\hat{x}).\label{3.1}
\end{align}
The rest of this section consists of three subsections devoted to determining the four expansion terms (coefficients).

\subsection{Equations for the expansion coefficients}\label{subsection3.1}
By referring to \cite{XX,Y2}, it is easy to see that the outer solution $(\bar{u},\bar{v})$ satisfies
\begin{gather}
\bar{u}_{t}+A_{11}\bar{u}_{x_1}+\sum_{j=2}^dA_{j11}\bar{u}_{x_j}=0,\label{3.2}\\[2mm]
\bar{v}=0.\label{3.3}
\end{gather}
For further discussion, we denote by $\mathcal{L}$ the differential operator 
$$
\mathcal{L}(U):=\partial_{t}U +\sum\limits_{j=1}^d A_j\partial_{x_j}U-QU/\epsilon.
$$

To determine the boundary-layer correction $(\mu,\nu)$ which satisfies 
\begin{equation}\label{tem3.4}
	(\mu_0,\nu_0)(t,\infty,\hat{x})=(\mu_1,\nu_1)(t,\infty,\hat{x})=(\mu_2,\nu_2)(t,\infty,\hat{x})=0,
\end{equation}
we notice that the approximation solution $U_\epsilon$ should make  
$$
\mathcal{L}(U_\epsilon)=\mathcal{L}\left({\begin{array}{*{20}c}
  \vspace{1.5mm}\bar{u}\\
  \bar{v}
\end{array}}\right)+\mathcal{L}\left({\begin{array}{*{20}c}
  \vspace{1.5mm}\mu\\
  \nu
\end{array}}\right)
$$
as small as possible. Note that the first term is
\begin{align} 
\mathcal{L}\left({\begin{array}{*{20}c}
  \vspace{1.5mm}\bar{u}\\
  \bar{v}
\end{array}}\right)=\left({\begin{array}{*{20}c}
  \vspace{1.5mm}0 \\
  A_{12}^*\partial_{x_1}\bar{u}+\sum\limits_{j=2}^d A_{j12}^*\partial_{x_j}\bar{u}
  \end{array}}\right)=\left({\begin{array}{*{20}c}
  \vspace{1.5mm}0 \\
  O(1)
  \end{array}}\right) \label{tem3.5}
\end{align}
due to \eqref{3.2} and \eqref{3.3}.
Thus we compute 
\begin{align*}
\mathcal{L}\left({\begin{array}{*{20}c}
  \vspace{1.5mm}\mu\\
  \nu
\end{array}}\right)
=&~\partial_t\left({\begin{array}{*{20}c}
  \vspace{1.5mm}\mu \\
  \nu
  \end{array}}\right)+
\frac{1}{\epsilon}\left({\begin{array}{*{20}c}
  \vspace{1.5mm}A_{11} & A_{12} \\
  A_{12}^* & A_{22}
  \end{array}}\right)
\partial_y\left({\begin{array}{*{20}c}
  \vspace{1.5mm}\mu \\
  \nu
  \end{array}}\right) 
+\frac{1}{\sqrt{\epsilon}}\left({\begin{array}{*{20}c}
  \vspace{1.5mm}A_{11} & A_{12} \\
  A_{12}^* & A_{22}
  \end{array}}\right)
\partial_z\left({\begin{array}{*{20}c}
  \vspace{1.5mm}\mu \\
  \nu
  \end{array}}\right) \\[2mm]
&+ \sum_{j=2}^d \left({\begin{array}{*{20}c}
  \vspace{1.5mm}A_{j11} & A_{j12} \\
  A_{j12}^* & A_{j22}
  \end{array}}\right)\partial_{x_j}\left({\begin{array}{*{20}c}
  \vspace{1.5mm}\mu \\
  \nu
  \end{array}}\right)
-\frac{1}{\epsilon}
\left({\begin{array}{*{20}c}
  \vspace{1.5mm}0 & 0 \\
  0 & S
  \end{array}}\right)
\left({\begin{array}{*{20}c}
  \vspace{1.5mm}\mu \\
  \nu
  \end{array}}\right). 
\end{align*}
By letting the coefficients of $\epsilon^{-1}$, $\epsilon^{-1/2}$ and $\epsilon^0$ in the above expression be zero, we obtain
\begin{eqnarray}\label{tem3.6}
\left\{{\begin{array}{*{20}l}
\epsilon^{-1}:~~\left({\begin{array}{*{20}c}
  \vspace{1.5mm}A_{11} & A_{12} \\
  A_{12}^* & A_{22}
  \end{array}}\right)\partial_y\left({\begin{array}{*{20}c}
  \vspace{1.5mm}\mu_0 \\
  \nu_0
  \end{array}}\right) =\left({\begin{array}{*{20}c}
  \vspace{1.5mm}0 & 0 \\
  0 & S
  \end{array}}\right)
  \left({\begin{array}{*{20}c}
  \vspace{1.5mm}\mu_0+\mu_1 \\
  \nu_0+\nu_1
  \end{array}}\right),\\[8mm]
\epsilon^{-1/2}:~\left({\begin{array}{*{20}c}
  \vspace{1.5mm}A_{11} & A_{12} \\
  A_{12}^* & A_{22}
  \end{array}}\right)\partial_z\left({\begin{array}{*{20}c}
  \vspace{1.5mm}\mu_1 \\
  \nu_1
  \end{array}}\right) =\left({\begin{array}{*{20}c}
  \vspace{1.5mm}0 & 0 \\
  0 & S
  \end{array}}\right)
  \left({\begin{array}{*{20}c}
  \vspace{1.5mm}\mu_2 \\
  \nu_2
  \end{array}}\right),\\[8mm]
  \epsilon^0:\qquad \partial_t\mu_{0}+\partial_t\mu_{1}+A_{11}\partial_z\mu_{2}+A_{12}\partial_z\nu_{2}\\
\qquad\qquad\qquad+\sum\limits_{j=2}^d\big[A_{j11}\partial_{x_j}(\mu_0+\mu_1)+A_{j12}\partial_{x_j}(\nu_0+\nu_1)\big]=0.
\end{array}}\right.
\end{eqnarray}
Here only the first $(n-r)$ rows of the coefficient of $\epsilon^0$ are considered, which is similar to \eqref{tem3.5}.
Motivated by \cite{XX}, we set  
\begin{align}
\nu_1=&~ 0. \label{tem3.7} 
\end{align}
Then the first two relations in \eqref{tem3.6} become
\begin{align}
\left({\begin{array}{*{20}c}
  \vspace{1.5mm}A_{11} & A_{12} \\
  A_{12}^* & A_{22}
  \end{array}}\right)\partial_y\left({\begin{array}{*{20}c}
  \vspace{1.5mm}\mu_0 \\
  \nu_0
  \end{array}}\right)=&\left({\begin{array}{*{20}c}
  \vspace{1.5mm}0 & 0 \\
  0 & S
  \end{array}}\right)
\left({\begin{array}{*{20}c}
  \vspace{1.5mm}\mu_0 \\
  \nu_0 
  \end{array}}\right),\label{tem3.8}\\[2mm]
A_{11}\partial_z\mu_1=&~ 0, \label{tem3.9}\\[2mm]
A_{12}^*\partial_z\mu_1=&~ S\nu_2.\label{tem3.10}
\end{align}

Having these equations, we show how to determine the expansion coefficients in \eqref{3.1}. Firstly, $(\mu_0,\nu_0)$ is simply determined by the ordinary differential equations \eqref{tem3.8}, which will be further discussed in Proposition \ref{prop3.3} below. 
In order to obtain $\mu_1$, we recall Assumption \ref{Asp2.1} that the symmetric matrix $A_{11}$ has $n_1^o$ zero eigenvalues and $(n-r-n_1^o)$ nonzero eigenvalues. Then there exist an $(n-r)\times (n-r-n_1^o)$-matrix $P_1$ and an $(n-r)\times n_1^o$-matrix $P_0$ such that $(P_1,P_0)$ is an orthonormal matrix satisfying
\begin{align}\label{tem3.11}
\left({\begin{array}{*{20}c}
  \vspace{1.5mm}P_1^* \\
  P_0^*
  \end{array}}\right)A_{11}(P_1, P_0)=\left({\begin{array}{*{20}c}
  \vspace{1.5mm}\Lambda_1 & 0 \\
  0 & 0
  \end{array}}\right),\qquad 
(P_1, P_0)\left({\begin{array}{*{20}c}
  \vspace{1.5mm}P_1^* \\
  P_0^*
  \end{array}}\right)=I_{n-r}.
\end{align} 
Here $\Lambda_1$ is an invertible diagonal $(n-r-n_1^o)\times (n-r-n_1^o)$-matrix. Multiplying \eqref{tem3.9} with $P_1^*$ from the left yields 
\begin{align*}
0=P_1^*A_{11}\partial_z\mu_1=P_1^*A_{11}(P_1P_1^*+P_0P_0^*)\partial_z\mu_1=\Lambda_1\partial_z(P_1^*\mu_1).
\end{align*}
Since $\mu_1(t,\infty,\hat{x})=0$, we have
\begin{equation}\label{tem3.12}
	P_1^*\mu_1\equiv0.
\end{equation}
Note that $\mu_1=P_1P_1^*\mu_1+P_0P_0^*\mu_1=P_0P_0^*\mu_1$.

In order to obtain $P_0^*\mu_1$, we use \eqref{tem3.7} and rewrite the third equation in \eqref{tem3.6} as
\begin{equation}
 \partial_t\mu_{1}+A_{11}\partial_z\mu_{2}+A_{12}\partial_z\nu_{2} 
 +\sum\limits_{j=2}^d A_{j11}\partial_{x_j} \mu_1 = G\label{tem3.13}
\end{equation}
with 
$$
G=G(\mu_0,\nu_0)=-\left[\partial_t\mu_0+\sum\limits_{j=2}^d(A_{j11}\partial_{x_j} \mu_0+A_{j12}\partial_{x_j} \nu_0)\right].
$$
Multiplying \eqref{tem3.13} with $P_0^*$ from the left and using \eqref{tem3.10}, we have 
\begin{equation*}
\partial_t(P_0^*\mu_1)+P_0^*A_{12}S^{-1}A_{12}^*\partial_{zz}\mu_1+\sum\limits_{j=2}^d P_0^*A_{j11}\partial_{x_j} \mu_1= P_0^*G. 
\end{equation*} 
Because $\mu_1=P_0P_0^*\mu_1+P_1P_1^*\mu_1=P_0P_0^*\mu_1$ due to \eqref{tem3.12}, the last equation can be written as
\begin{align}\label{tem3.14}
\partial_t(P_0^*\mu_1)+\left[P_0^*A_{12}S^{-1}A_{12}^*P_0\right]\partial_{zz}(P_0^*\mu_1)+\sum\limits_{j=2}^d (P_0^*A_{j11}P_0) \partial_{x_j} (P_0^*\mu_1)=P_0^*G.
\end{align}
This differential equation will be further discussed in Proposition \ref{prop3.2} below. Once it is solved, we obtain $\mu_1$ and $\nu_2$ by \eqref{tem3.10} for $S$ is invertible. 

It remains to determine $\mu_2$. Multiplying \eqref{tem3.13} with $P_1^*$ from the left and using \eqref{tem3.12} yield
\begin{align*}
P_1^*G=&~ P_1^*A_{11}\partial_z \mu_2+ P_1^*A_{12}S^{-1}A_{12}^*\partial_{zz}\mu_1+\sum\limits_{j=2}^d P_1^*A_{j11}\partial_{x_j} \mu_1 \\[1mm]
 =&~\Lambda_1P_1^*\partial_z \mu_2 + P_1^*A_{12}S^{-1}A_{12}^*\partial_{zz}\mu_1+\sum\limits_{j=2}^d P_1^*A_{j11}\partial_{x_j} \mu_1.
\end{align*}
With \eqref{tem3.4}, we integrate the last equation to obtain
\begin{equation}\label{tem3.15}
	P_1^* \mu_2 = -\Lambda_1^{-1}\left[ P_1^*A_{12}S^{-1}A_{12}^*\partial_{z}\mu_1 - \sum\limits_{j=2}^d P_1^*A_{j11} \int_{z}^{\infty} \partial_{x_j} \mu_1+\int_{z}^{\infty}P_1^*G\right].
\end{equation}
This gives $P_1^*\mu_2$ once $(\mu_0,\nu_0)$ and $\mu_1$ are determined. Note that it is not necessary for our purpose to determine $\mu_2$ itself. 

In summary, the expansion coefficients in \eqref{3.1} are determined in the following way: 
\begin{eqnarray*}
\left\{{\begin{array}{*{20}l}
  \text{\eqref{3.3}}~\Rightarrow~\bar{v}=0 ,\\[3mm]
  \text{\eqref{tem3.7}}~\Rightarrow~\nu_1=0,\\[3mm]
  \text{\eqref{tem3.12}}\Rightarrow P_1^*\mu_1=0,
\end{array}}\right.\qquad\left\{{\begin{array}{*{20}l}
  \text{\eqref{3.2}}~\Rightarrow~\bar{u},\\[3mm]
  \text{\eqref{tem3.8}}~\Rightarrow~(\mu_0,\nu_0),\\[3mm]
  \text{\eqref{tem3.14}}\Rightarrow~P_0^*\mu_1,
\end{array}}\right.\qquad
\left\{{\begin{array}{*{20}l}
  \text{\eqref{tem3.15}}\Rightarrow~P_1^*\mu_2 ,\\[3mm]
  \text{\eqref{tem3.10}}\Rightarrow~\nu_2.
\end{array}}\right.
\end{eqnarray*}
The second set of equations \eqref{3.2}, \eqref{tem3.8}, and \eqref{tem3.14} will be further discussed in the following two subsections.

\subsection{Solvability}\label{subsection3.2}
Now we analyze the solvability of the equations \eqref{3.2}, \eqref{tem3.8} and \eqref{tem3.14}. 
For \eqref{3.2}, we refer to Section \ref{section2} and state 
\begin{prop}\label{prop3.1}
The system \eqref{3.2} is symmetrizable hyperbolic, the coefficient $A_{11}$ has $n_1^+$ positive eigenvalues, and zero is an eigenvalue of $A_{11}$ with multiplicity $n_1^o$. 
\end{prop}

Next we turn to the equations \eqref{tem3.14} for the $n_1^o$ unknowns $P_0^*\mu_1$. Referring to the coefficient of $\partial_{zz}(P_0^*\mu_1)$, we set
\begin{equation*}
	K = A_{12}^*P_0
\end{equation*}
and the coefficient matrix can be written as $K^*S^{-1}K$. 
For these matrices, we have

\begin{prop}\label{prop3.2}
The $r\times n_1^o$-matrix $K$ is of full-column-rank and the relation $r\geq n_1^o$ holds.
Moreover, the coefficient matrix $K^*S^{-1}K$ in \eqref{tem3.14} is a symmetric negative definite matrix. 
\end{prop}

\begin{proof}
Recall the orthonormal $(n-r)\times(n-r)$-matrix $(P_1,P_0)$ defined in \eqref{tem3.11}. We introduce
\begin{align}
T:=&\left({\begin{array}{*{20}c}
  \vspace{1.5mm}P_1 & P_0 & 0 \\
  0 & 0 & I_r
  \end{array}}\right)\left({\begin{array}{*{20}c}
  \vspace{1.5mm}I_{n-r-n_1^o} & 0 & -\Lambda_{1}^{-1}P_1^*A_{12} \\
  \vspace{1.5mm}0 & I_{n_1^o} & 0\\
  0 & 0 & I_r 
  \end{array}}\right) \nonumber\\[2mm]
   =& \left({\begin{array}{*{20}c}
  \vspace{1.5mm}P_1 & P_0 & -P_1\Lambda_{1}^{-1}P_1^*A_{12} \\
  0 & 0 & I_r
  \end{array}}\right) \label{tem3.16}
\end{align}
and consider the congruent transformation of the coefficient matrix 
$A_1=\left({\begin{array}{*{20}c}
  \vspace{1.5mm}A_{11} & A_{12} \\
  A_{12}^* & A_{22}
  \end{array}}\right)$:

\begin{align}
&T^*\left({\begin{array}{*{20}c}
  \vspace{1.5mm}A_{11} & A_{12} \\
  A_{12}^* & A_{22}
  \end{array}}\right)T \nonumber\\[2mm]
  =&\left({\begin{array}{*{20}c}
  \vspace{1.5mm}I & 0 & 0 \\
  \vspace{1.5mm}0 & I & 0\\
  -A_{12}^*P_1\Lambda_{1}^{-1} & 0 & I
  \end{array}}\right)\left({\begin{array}{*{20}c}
  \vspace{1.5mm}\Lambda_{1} & 0 & P_1^*A_{12} \\
  \vspace{1.5mm}0 & 0 & P_0^*A_{12}\\
  A_{12}^*P_1 & A_{12}^*P_0 & A_{22}
  \end{array}}\right)\left({\begin{array}{*{20}c}
  \vspace{1.5mm}I & 0 & -\Lambda_{1}^{-1}P_1^*A_{12} \\
  \vspace{1.5mm}0 & I & 0\\
  0 & 0 & I
  \end{array}}\right) \nonumber\\[2mm]
  =&\left({\begin{array}{*{20}c}
  \vspace{1.5mm}\Lambda_{1} & 0 & 0 \\
  \vspace{1.5mm}0 & 0 & K^*\\
  0 & K & A
  \end{array}}\right)\quad  \text{with} ~~A=A_{22}-A_{12}^*P_1\Lambda_1^{-1}P_1^*A_{12}.\label{tem3.17}
\end{align}
Thus it follows from the invertibility of $A_1$ that the matrix 
$\left({\begin{array}{*{20}c}
  \vspace{1.5mm}0 & K^*\\
   K & A
  \end{array}}\right)$ is invertible. This particularly implies that the $r\times n_1^o$-matrix $K$ must be of full-column-rank 
and $r\geq n_1^o$. Moreover, $K^*S^{-1}K$ is a symmetric negative definite matrix since so is $S$. This completes the proof.
\end{proof}

\begin{remark}
The inequality $r\geq n_1^o$ can also be deduced from the invertibility of $A_1$ and the well-known sub-characteristic condition for hyperbolic relaxation systems \cite{CLL,Y7}. We leave the proof to the interested reader. 
\end{remark}

At last we discuss global solutions to the ordinary differential equations \eqref{tem3.8} defined for $y\in [0,\infty)$ and have the following facts. 


\begin{prop}\label{prop3.3}
Under Assumption \ref{Asp2.1}, there exist an $(n-r)\times (r-n_1^o)$-matrix $N$ (defined in \eqref{tem3.23} below) and an $r \times (r-n_1^o)$-matrix $\tilde{K}$ (defined in \eqref{tem3.20}) such that the solution $(\mu_0,\nu_0)$ of \eqref{tem3.8} can be expressed by
$$
\mu_0=Nw,\qquad \quad \nu_0=\tilde{K}w,
$$
and $w$ satisfies the ordinary differential equations
\begin{equation}\label{tem3.18}
	\quad\partial_yw=(\tilde{K}^*A\tilde{K})^{-1}(\tilde{K}^*S\tilde{K})w.
\end{equation}
Moreover, the matrix $(\tilde{K}^*A\tilde{K})^{-1}(\tilde{K}^*S\tilde{K})$ is invertible with $(n^+-n_1^o-n_1^+)$ negative eigenvalues.
\end{prop}

\begin{remark}
From this proposition, we see that the bounded solution $w=w(t,y,\hat{x})\equiv 0$ if $n^+=n_1^o+n_1^+$. In this situation, we have $(\mu_0,\nu_0)=(\mu_0,\nu_0)(t,y,\hat{x})\equiv 0$. Namely, the three scales in the expansion \eqref{3.1} degenerate to two scales.
\end{remark}

\begin{proof}
We divide the proof into three steps. 

\textbf{Step 1:} With the congruent transformation \eqref{tem3.17}, we rewrite \eqref{tem3.8} as
\begin{align*} 
T^*A_1TT^{-1}\left({\begin{array}{*{20}c}
  \vspace{1.5mm}\mu_0 \\
  \nu_0
  \end{array}}\right)_y=\left({\begin{array}{*{20}c}
  \vspace{1.5mm}\Lambda_{1} & 0 & 0 \\
  \vspace{1.5mm}0 & 0 & K^*\\
  0 & K & A
  \end{array}}\right)T^{-1}\left({\begin{array}{*{20}c}
  \vspace{1.5mm}\mu_0 \\
  \nu_0
  \end{array}}\right)_y=T^*\left({\begin{array}{*{20}c}
  \vspace{1.5mm}0_{n-r} & 0 \\
  0 & S
  \end{array}}\right)TT^{-1}\left({\begin{array}{*{20}c}
  \vspace{1.5mm}\mu_0 \\
  \nu_0
  \end{array}}\right).
\end{align*}
By using the expression of $T$ in \eqref{tem3.16}, it is easy to compute
$$
T^*\left({\begin{array}{*{20}c}
  \vspace{1.5mm}0_{n-r} & 0 \\
  0 & S
  \end{array}}\right)T = \left({\begin{array}{*{20}c}
  \vspace{1.5mm}0_{n-r-n_1^o} & 0 & 0 \\
  \vspace{1.5mm}0 & 0_{n_1^o} & 0\\
  0 & 0 & S
  \end{array}}\right),\qquad T^{-1}=\left({\begin{array}{*{20}c}
  \vspace{1.5mm}P_1^* & \Lambda_{1}^{-1}P_1^*A_{12} \\
  \vspace{1.5mm}P_0^* & 0 \\
                  0 & I_r 
  \end{array}}\right). 
$$
Thus the original differential equations become 
\begin{align*}
\left({\begin{array}{*{20}c}
  \vspace{1.5mm}\Lambda_{1} & 0 & 0 \\
  \vspace{1.5mm}0 & 0 & K^*\\
  0 & K & A
  \end{array}}\right)\left({\begin{array}{*{20}c}
  \vspace{1.5mm}P_1^*\mu_0+\Lambda_1^{-1}P_1^*A_{12}\nu_0 \\
  \vspace{1.5mm} P_0^*\mu_0\\
                 \nu_0
  \end{array}}\right)_y=&\left({\begin{array}{*{20}c}
  \vspace{1.5mm}0 & 0 & 0 \\
  \vspace{1.5mm}0 & 0 & 0\\
  0 & 0 & S
  \end{array}}\right)\left({\begin{array}{*{20}c}
  \vspace{1.5mm} P_1^*\mu_0+\Lambda_1^{-1}P_1^*A_{12}\nu_0 \\
  \vspace{1.5mm} P_0^*\mu_0\\
                 \nu_0
  \end{array}}\right).
\end{align*}
Since $\mu_0(t,\infty,\hat{x})=0$ and $\nu_0(t,\infty,\hat{x})=0$, it follows immediately from the last equation that
\begin{eqnarray}\label{tem3.19}
\left\{{\begin{array}{*{20}l}
  ~P_1^*\mu_0 =-\Lambda_1^{-1}P_1^*A_{12}\nu_0,\\[3mm]
  ~K^*\nu_0 = 0,\\[3mm]
  ~KP_0^*\mu_{0y}+A\nu_{0y}=S\nu_0. 
\end{array}}\right.
\end{eqnarray}

\textbf{Step 2:}
Recall from Proposition \ref{prop3.2} that $K$ is an $r\times n_1^o$ full-column-rank matrix. Then there exists an $r\times (r-n_1^o)$ full-rank matrix $\tilde{K}$ such that 
\begin{equation}\label{tem3.20}
	\big(K,~\tilde{K}\big) \quad \text{is invertible} \quad \text{and} \quad \tilde{K}^*K=0.
\end{equation}
Thus we deduce from $K^*\nu_0=0$ in \eqref{tem3.19} that 
\begin{align}
	\nu_0=K(K^*K)^{-1}K^*\nu_0+\tilde{K}(\tilde{K}^*\tilde{K})^{-1}\tilde{K}^*\nu_0 
		 =\tilde{K}(\tilde{K}^*\tilde{K})^{-1}\tilde{K}^*\nu_0 
		 \equiv\tilde{K}w\label{tem3.21}
\end{align}
with $w=(\tilde{K}^*\tilde{K})^{-1}\tilde{K}^*\nu_0$.
By \eqref{tem3.20} and \eqref{tem3.21}, we multiply the third equation in \eqref{tem3.19} with $\tilde{K}^*$ from the left to obtain
\begin{align}\label{tem3.22}
\tilde{K}^*A\tilde{K}w_y=\tilde{K}^*S\tilde{K}w.
\end{align}
If the matrix $\tilde{K}^*A\tilde{K}$ is invertible, we can obtain the ordinary differential equations \eqref{tem3.18}. This invertibility will be shown in Step 3.

To express $\mu_0$ in term of $w$, we multiply the third equation in \eqref{tem3.19} with $K^*$ from the left and use \eqref{tem3.21} to obtain
$$
(K^*K)(P_0^*\mu_0)_y+K^*A\tilde{K}w_y=K^*S\tilde{K}w.
$$
Using \eqref{tem3.22} and the invertibility of $\tilde{K}^*S\tilde{K}$, we have
$$
(K^*K)(P_0^*\mu_0)_y+K^*A\tilde{K}w_y=(K^*S\tilde{K})(\tilde{K}^*S\tilde{K})^{-1}(\tilde{K}^*A\tilde{K})w_y.
$$
Then it follows from $\mu_0(t,\infty,\hat{x})=0$ and $w(t,\infty,\hat{x})=(\tilde{K}^*\tilde{K})^{-1}\tilde{K}^*\nu_0(t,\infty,\hat{x})=0$ that 
\begin{equation*}
P_0^*\mu_0=(K^*K)^{-1}\left[(K^*S\tilde{K})(\tilde{K}^*S\tilde{K})^{-1}(\tilde{K}^*A\tilde{K})-(K^*A\tilde{K})\right]w.
\end{equation*}
This, together with the first equation in \eqref{tem3.19} and \eqref{tem3.21}, gives
\begin{equation*}
	\mu_0 = P_1(P_1^*\mu_0) + P_0(P_0^*\mu_0)= Nw,
\end{equation*}
where
\begin{equation}\label{tem3.23}
	N:=-P_1\Lambda_1^{-1}P_1^*A_{12}\tilde{K} + P_0
  (K^*K)^{-1}\left[(K^*S\tilde{K})(\tilde{K}^*S\tilde{K})^{-1}(\tilde{K}^*A\tilde{K})-(K^*A\tilde{K})\right].
\end{equation}


\textbf{Step 3:} 
It remains to show that $\tilde{K}^*A\tilde{K}$ is invertible and $(\tilde{K}^*A\tilde{K})^{-1}(\tilde{K}^*S\tilde{K})$ has $(n^+-n_1^o-n_1^+)$ negative eigenvalues.
To do this, we introduce
\begin{equation}\label{tem3.24}
\bar{L}_1 := \left({\begin{array}{*{20}c}
\vspace{1.5mm}I_{n_1^o} & 0 & 0 \\
0 & K & \tilde{K}
\end{array}}\right),\quad\bar{L}_2:=\left({\begin{array}{*{20}c}
\vspace{2mm}I_{n_1^o} & 0 & -(K^*K)^{-1}(K^*A\tilde{K}) \\
\vspace{2mm}0 & I_{n_1^o} & 0 \\
0 & 0 & I_{r-n_1^o}
\end{array}}\right)
\end{equation}
and consider the following congruent transformation 
\begin{align*}
\bar{L}_2^*\bar{L}_1^*\left({\begin{array}{*{20}c}
\vspace{1.5mm}0 & K^* \\
K & A
\end{array}}\right)\bar{L}_1\bar{L}_2
=&
\bar{L}_2^*
\left({\begin{array}{*{20}c}
\vspace{1.5mm}0 & K^*K & 0 \\
\vspace{1.5mm}K^*K & K^*AK & K^*A\tilde{K} \\
0 & \tilde{K}^*AK & \tilde{K}^*A\tilde{K}
\end{array}}\right)
\bar{L}_2 
= \left({\begin{array}{*{20}c}
\vspace{1.5mm}0 & K^*K & 0 \\
\vspace{1.5mm}K^*K & K^*AK & 0 \\
0 & 0 & \tilde{K}^*A\tilde{K}
\end{array}}\right).
\end{align*}
With this and the congruent transformation \eqref{tem3.17}, we find that 
\begin{equation}\label{tem3.25}
L^*A_1L=
\left({\begin{array}{*{20}c}
  \vspace{1.5mm}\Lambda_{1} & 0 & 0 & 0\\
  \vspace{1.5mm}0 & 0  & K^*K & 0 \\
  \vspace{1.5mm}0 & K^*K & K^*AK & 0 \\
0 & 0 & 0 & \tilde{K}^*A\tilde{K}
  \end{array}}\right)\quad \text{with}~~L:=T\left({\begin{array}{*{20}c}
  \vspace{1.5mm}I_{n-r-n_1^o} & 0 \\
                0 & \bar{L}_1\bar{L}_2
  \end{array}}\right).
\end{equation}
Since $A_1$ is invertible, we see from \eqref{tem3.25} that $\tilde{K}^*A\tilde{K}$ is invertible. 

Moreover, we recall the definition of $n^+$ and $n_1^+$ in Section \ref{section2}. Then the number of positive eigenvalues of $\tilde{K}^*A\tilde{K}$ equals to
$$
n^+-n_1^+-\text{the number of positive eigenvalues of } \left({\begin{array}{*{20}c}
\vspace{2mm}0 & K^*K \\
K^*K  & K^*AK 
\end{array}}\right).
$$
In Appendix \ref{AppendA1}, we show that the matrix $\left({\begin{array}{*{20}c}
\vspace{1mm}0 & K^*K \\
K^*K  & K^*AK 
\end{array}}\right)$ has $n_1^o$ positive eigenvalues. Therefore, $\tilde{K}^*A\tilde{K}$ has $(n^+-n_1^+-n_1^o)$ positive eigenvalues. Since $\tilde{K}^*S\tilde{K}$ is negative definite, the matrix $(\tilde{K}^*A\tilde{K})^{-1}(\tilde{K}^*S\tilde{K})$ in \eqref{tem3.18} has $(n^+-n_1^+-n_1^o)$ negative eigenvalues. This completes the proof.
\end{proof}

\subsection{Boundary conditions for the expansion coefficients}\label{subsection3.3}

Here we show how to determine an initial condition for \eqref{tem3.8} and boundary conditions for \eqref{3.2} and \eqref{tem3.14}. 
The approximate solution in \eqref{3.1} is constructed to satisfy 
\begin{eqnarray}\label{tem3.26}
  (B_u,B_v)\left({\begin{array}{*{20}c}
  \vspace{1.5mm}\bar{u} +\mu_0+\mu_1 \\
  \bar{v} +\nu_0+\nu_1
  \end{array}}\right)(t,0,\hat{x})=b(t,\hat{x}).
\end{eqnarray}
Referring to \eqref{3.3},\eqref{tem3.7} and \eqref{tem3.12}, we can write \eqref{tem3.26} as
\begin{eqnarray*} 
  (B_u,B_v)\left({\begin{array}{*{20}c}
  \vspace{1.5mm}\bar{u} +\mu_0+P_0(P_0^*\mu_1) \\
  \nu_0
  \end{array}}\right)(t,0,\hat{x})=b(t,\hat{x}).
\end{eqnarray*}
From Proposition \ref{prop3.3}, we know that $\mu_0=Nw$ and $\nu_0=\tilde{K}w$. Thus the last equation can be rewritten as 
\begin{equation}\label{tem3.27}
	(B_u, ~B_uP_0, ~B_uN+B_v\tilde{K})\left({\begin{array}{*{20}c}
  \vspace{1.5mm}\bar{u}\\
  \vspace{1.5mm}P_0^*\mu_1\\
                w
  \end{array}}\right)(t,0,\hat{x})=b(t,\hat{x}).
\end{equation}

To obtain a boundary condition for the hyperbolic system \eqref{3.2}, we express $\bar{u}=\bar{u}(t,0,\hat{x})$ in terms of its characteristic variables. 
Recall that $P_0$ consists of eigenvectors of $A_{11}$ associated with the zero eigenvalue. Let $P_1^S$ and $P_1^U$ be the respective right-stable and right-unstable matrices of $A_{11}$. Then we have the decomposition
\begin{equation}\label{tem3.28}
	\bar{u}(t,0,\hat{x}) =P_1^S \alpha^S +P_1^U \alpha^U+ P_0 \alpha^0.  
\end{equation}
Here vectors $\alpha^U$ and $\alpha^S$ are the incoming and outgoing 
modes for the equilibrium system, while $\alpha^0$ is the characteristic variable associated with the zero eigenvalue.
By the classical theory of hyperbolic IBVPs \cite{GKO}, the boundary condition for $\bar{u}$ should be an expression of $\alpha^U$ in terms of $\alpha^0$ and $\alpha^S$. Thus we rewrite \eqref{tem3.27} as 
\begin{eqnarray}\label{tem3.29}
 	(B_uP_1^U, ~B_uP_0, ~B_uN+B_v\tilde{K})\left({\begin{array}{*{20}c}
  \vspace{1.5mm}\alpha^U\\
  \vspace{1.5mm}P_0^*\mu_1\\
                w
  \end{array}}\right) =b(t,\hat{x})-B_uP_1^S\alpha^S-B_uP_0  \alpha^0.
\end{eqnarray}

On the other hand, recall Proposition \ref{prop3.3} that $w$ satisfies the ordinary differential equations
$$
\partial_yw=(\tilde{K}^*A\tilde{K})^{-1}(\tilde{K}^*S\tilde{K})w.
$$ 
In order for $w$ to be bounded, its initial value should satisfy 
\begin{equation}\label{tem3.30}
	L_2^Uw(t,0,\hat{x})=0,
\end{equation}
where $L_2^U$ is the left-unstable matrix of $(\tilde{K}^*A\tilde{K})^{-1}(\tilde{K}^*S\tilde{K})$. By Proposition \ref{prop3.3}, the invertible matrix $(\tilde{K}^*A\tilde{K})^{-1}(\tilde{K}^*S\tilde{K})$ has $(r-n_1^o)-(n^+ - n_1^o - n_1^+)= n_1^+ + r- n^+$ positive eigenvalues. Thus $L_2^U$ is an $(n_1^++r-n^+)\times (r-n_1^o)$-matrix. 

Combining the equations \eqref{tem3.29} and \eqref{tem3.30}, we arrive at
\begin{align}\label{tem3.31}
\left({\begin{array}{*{20}c}
  \vspace{1.5mm}B_uP_1^U & B_uP_0 & B_uN+B_v\tilde{K} \\
  0 & 0 & L_2^U
  \end{array}}\right)\left({\begin{array}{*{20}c}
  \vspace{1.5mm}\alpha^U \\
  \vspace{1.5mm}P_0^*\mu_1 \\
  w
  \end{array}}\right) =
  \left({\begin{array}{*{20}c}
  \vspace{1.5mm}b - B_u P_1^S \alpha^S- B_u P_0 \alpha^0 \\
  0
  \end{array}}\right). 
\end{align}
This is a system of $(n_1^++r)$ linear algebraic equations for $(n_1^++r)$ variables 
 $\alpha^U$, $P_0^*\mu_1$ and $w$, which have $n_1^+$, $n_1^o$ and $(r-n_1^o)$ components respectively.
If the square matrix in \eqref{tem3.31} 
is invertible, we can solve the system of linear algebraic equations to obtain the boundary conditions for \eqref{3.2} and \eqref{tem3.14} as well as $w(t,0,\hat{x})$. 

In the next section, we will show the invertibility of the square matrix in \eqref{tem3.31} under the GKC.

\section{Reduced boundary conditions}\label{section4}
In this section, we show that the square matrix in \eqref{tem3.31} is invertible under the GKC and explicitly write down the boundary condition for the equilibrium system \eqref{3.2}. Moreover, we show that the hyperbolic system \eqref{3.2} together with the derived boundary condition is well-posed in the sense of \cite{MO}.


For this purpose, we follow \cite{Y2} and compute the right-stable matrix $R_M^S(\xi,\omega,\eta)$ for $M(\xi,\omega,\eta)$ with $\eta$ sufficiently large. The invertibility will be seen from the GKC as $\eta\rightarrow \infty$. 
The computation consists of a number of similarity transformations for $M(\xi,\omega,\eta)$.
For convenience, we introduce the notation
$$
C(\omega)= i \sum_{j=2}^d\omega_j A_{j11}~~\text{and}~~C_{kl}(\omega) = P_k^*C(\omega)P_l,\quad \text{for}~~k,l=0,1
$$
with $P_0$ and $P_1$ defined in \eqref{tem3.11}. 

The rest of this section consists of three subsections.

\subsection{Preparations}
In this subsection, we present some similarity transformations of the matrix $M=M(\xi,\omega,\eta)$.
\begin{lemma}\label{lemma4.1}
Under Assumption \ref{Asp2.1}, the matrix $M=M(\xi,\omega,\eta)$ is similar to 
\begin{equation}\label{4.1}
 \bar{M}=\left({\begin{array}{*{20}c}
  \vspace{1.5mm} M_{1}(\xi,\omega)& \star & \star & \quad\star\\
  \vspace{1.5mm} \star & \star & \eta M_{3} +  \star & \quad\star\\
  \vspace{1.5mm} 0  &  M_2(\xi,\omega) & \star &\quad \star\\
  \star & \star & \star & \eta M_4 + \star
  \end{array}}\right) 
\end{equation}
with 
\begin{align*}
M_1(\xi,\omega) =&-\Lambda_1^{-1}\bigg(\big[\xi I +C_{11}(\omega)\big]-C_{10}(\omega)\big[\xi I + C_{00}(\omega)\big]^{-1}C_{01}(\omega) \bigg),\\ 
M_2(\xi,\omega) =&-(K^*K)^{-1}\big[\xi I + C_{00}(\omega)\big],\\ 
M_3 =&~(K^*K)^{-1}\left[K^*SK - (K^*S\tilde{K})(\tilde{K}^*S\tilde{K})^{-1}(\tilde{K}^*SK)\right],\\ 
M_4 = &~(\tilde{K}^*A\tilde{K})^{-1}(\tilde{K}^*S\tilde{K}).
\end{align*}
Here the notation $\star$ means a matrix independent of the parameter $\eta$.
\end{lemma}
\begin{proof}
With the matrix $L$ defined in \eqref{tem3.25}, we consider the similarity transformation
\begin{align}
	L^{-1}ML=& (L^*A_1L)^{-1}L^*\left(\eta Q -\xi I- i \sum_{j=2}^d\omega_jA_j\right)L.\label{4.2}
\end{align}
Because  
$$
\left({\begin{array}{*{20}c}
\vspace{2mm}0 & K^*K \\
K^*K  & K^*AK 
\end{array}}\right)^{-1}=\left({\begin{array}{*{20}c}
\vspace{2mm}-(K^*K)^{-1}(K^*AK)(K^*K)^{-1} & (K^*K)^{-1} \\
             (K^*K)^{-1} & 0 
\end{array}}\right),
$$
it follows from \eqref{tem3.25} that 
\begin{equation}\label{4.3}
	(L^*A_1L)^{-1}=\left({\begin{array}{*{20}c}
  \vspace{2mm}\Lambda_{1}^{-1} & 0 & 0 & 0\\
  \vspace{2mm}0 & -(K^*K)^{-1}(K^*AK)(K^*K)^{-1} & (K^*K)^{-1} & 0 \\
  \vspace{2mm}0 & (K^*K)^{-1} & 0 & 0 \\
0 & 0 & 0 & (\tilde{K}^*A\tilde{K})^{-1}
  \end{array}}\right).
\end{equation}

On the other hand, from \eqref{tem3.16}, \eqref{tem3.24} and \eqref{tem3.25} we have
\begin{align}
L \equiv & T \left({\begin{array}{*{20}c}
  \vspace{1.5mm}I & 0 \\
                0 & \bar{L}_1\bar{L}_2
  \end{array}}\right) \nonumber\\[2mm]
= & \left({\begin{array}{*{20}c}
  \vspace{1.5mm}P_1 & P_0 & -P_1\Lambda_{1}^{-1}P_1^*A_{12} \\
  0 & 0 & I_r
  \end{array}}\right)\left({\begin{array}{*{20}c}
\vspace{1.5mm}I & 0 & 0 & 0 \\
\vspace{1.5mm}0 & I & 0 & -(K^*K)^{-1}(K^*A\tilde{K}) \\
0 & 0 & K & \tilde{K}
\end{array}}\right) \nonumber \\[2mm]
=&\left({\begin{array}{*{20}c}
  \vspace{1.5mm}P_1 & P_0 & -P_1\Lambda_{1}^{-1}P_1^*A_{12}K & -P_1\Lambda_{1}^{-1}P_1^*A_{12}\tilde{K}-P_0(K^*K)^{-1}(K^*A\tilde{K})\\
  0 & 0 & K & \tilde{K}
  \end{array}}\right).\label{4.4}
\end{align}
Then we can compute (omit some details by $\star$)
\begin{align*}
L^*\left(\eta Q -\xi I-i \sum_{j=2}^d \omega_jA_j\right)L = &\left({\begin{array}{*{20}c}
  \vspace{1.5mm}P_1^* & 0\\
  \vspace{1.5mm}P_0^* & 0\\
  \vspace{1.5mm}\star & K^*\\
  \star & \tilde{K}^* 
  \end{array}}\right)\left({\begin{array}{*{20}c}
  \vspace{1.5mm}-\xi I_{n-r}-C(\omega) & \star\\
  \star & \eta S + \star
  \end{array}}\right)
  \left({\begin{array}{*{20}c}
  \vspace{1.5mm}P_1 & P_0 & \star & \star\\
  0 & 0 & K & \tilde{K}
  \end{array}}\right)\\[2mm]
  = &\left({\begin{array}{*{20}c}
  \vspace{1.5mm}-\xi P_1^*-P_1^*C(\omega) & \star \\
  \vspace{1.5mm}-\xi P_0^*-P_0^*C(\omega) & \star \\
  \vspace{1.5mm}\star & \eta K^*S + \star  \\
  \star & \eta \tilde{K}^*S + \star 
  \end{array}}\right)\left({\begin{array}{*{20}c}
  \vspace{1.5mm}P_1 & P_0 & \star & \star\\
  0 & 0 & K & \tilde{K}
  \end{array}}\right)\\[2mm]
  = &\left({\begin{array}{*{20}c}
  \vspace{1.5mm}-\xi I -C_{11}(\omega) & -C_{10}(\omega) & \star & \star \\
  \vspace{1.5mm}-C_{01}(\omega) & -\xi I -C_{00}(\omega) & \star & \star\\
  \vspace{1.5mm}\star & \star & \eta K^*SK + \star & \eta K^*S\tilde{K} + \star\\
  \star &  \star & \eta \tilde{K}^*SK + \star & \eta \tilde{K}^*S\tilde{K} + \star 
  \end{array}}\right).
\end{align*}
With this and \eqref{4.3}, the matrix $L^{-1}ML$ in \eqref{4.2} can be written as
\begin{align*}
\left({\begin{array}{*{20}c}
  \vspace{1.5mm}- \Lambda_1^{-1}[\xi I + C_{11}(\omega)] & - \Lambda_1^{-1} C_{10}(\omega) & \star & \star \\
  \vspace{1.5mm}\star & \star &  \eta (K^*K)^{-1}K^*SK + \star & \eta (K^*K)^{-1}K^*S\tilde{K} + \star\\
  \vspace{1.5mm}-(K^*K)^{-1}C_{01}(\omega) & -(K^*K)^{-1}[\xi I + C_{00}(\omega)] & \star & \star\\
  \star &  \star & \eta (\tilde{K}^*A\tilde{K})^{-1}\tilde{K}^*SK + \star & \eta (\tilde{K}^*A\tilde{K})^{-1}\tilde{K}^*S\tilde{K} + \star 
  \end{array}}\right).
\end{align*}

To eliminate $\eta (K^*K)^{-1}K^*S\tilde{K}$ at the position $(2,4)$ and $\eta (\tilde{K}^*A\tilde{K})^{-1}\tilde{K}^*SK$ at the position $(4,3)$ of the last block-matrix, we introduce
\begin{equation}\label{4.5}
	L_0 = \left({\begin{array}{*{20}c}
  \vspace{1.5mm} I_{n-r-n_1^o} &  0 & 0 &  0 \\
  \vspace{1.5mm}0 & I_{n_1^o} & 0 & (K^*K)^{-1}(K^*S\tilde{K})(\tilde{K}^*S\tilde{K})^{-1}(\tilde{K}^*A\tilde{K}) \\
  \vspace{1.5mm}0 & 0 & I_{n_1^o} & 0\\
  0 &  0 & -(\tilde{K}^*S\tilde{K})^{-1}(\tilde{K}^*SK) & I_{r-n_1^o} 
  \end{array}}\right) 
\end{equation}
and compute 
\begin{equation*} 
L_0^{-1}L^{-1}MLL_0=\left({\begin{array}{*{20}c}
  \vspace{1.5mm}- \Lambda_1^{-1}[\xi I + C_{11}(\omega)] & - \Lambda_1^{-1} C_{10}(\omega) & \star & \star \\
  \vspace{1.5mm}\star & \star &  \eta M_3 + \star & \star\\
  \vspace{1.5mm}-(K^*K)^{-1}C_{01}(\omega) & -(K^*K)^{-1}[\xi I + C_{00}(\omega)]  & \star & \star\\
  \star &  \star & \star & \eta M_4 + \star 
  \end{array}}\right).	
\end{equation*}
Here we have used the expressions of $M_3$ and $M_4$ below \eqref{4.1}. 

At last, we introduce 
\begin{equation}\label{4.6}
	L_1 = \left({\begin{array}{*{20}c}
  \vspace{1.5mm} I_{n-r-n_1^o} &  0 & 0 &  0 \\
  \vspace{1.5mm}-[\xi I_{n_1^o} +C_{00}(\omega)]^{-1}C_{01}(\omega) & I_{n_1^o} & 0 &  0 \\
  \vspace{1.5mm}0 & 0 & I_{n_1^o} & 0\\
  0 &  0 & 0 & I_{r-n_1^o} 
  \end{array}}\right) 
\end{equation}
to compute 
\begin{equation}\label{tem4.7}
\bar{M}=L_1^{-1}L_0^{-1}L^{-1}MLL_0L_1= \left({\begin{array}{*{20}c}
  \vspace{1.5mm} M_{1}(\xi,\omega)& \star & \star & \quad\star\\
  \vspace{1.5mm} \star & \star & \eta M_{3} +  \star & \quad\star\\
  \vspace{1.5mm} 0  &  M_2(\xi,\omega) & \star &\quad \star\\
  \star & \star & \star & \eta M_4 + \star
  \end{array}}\right) .	
\end{equation}
The proof is complete.
\end{proof}

For further reference, we use \eqref{4.4}-\eqref{4.6} to compute
\begin{equation}\label{tem4.8}
  LL_0L_1= \left({\begin{array}{*{20}c}
  \vspace{1.5mm}P_1-P_0[\xi I_{n_1^o} +C_{00}(\omega)]^{-1}C_{01}(\omega) & P_0 & \star & N\\
  0 & 0 & K-\tilde{K}(\tilde{K}^*S\tilde{K})^{-1}(\tilde{K}^*SK) & \tilde{K}
  \end{array}}\right). 
\end{equation}
Here $N$ is the matrix defined in \eqref{tem3.23}.

Having Lemma \ref{lemma4.1}, we establish 

\begin{lemma}\label{lemma4.2}
For sufficiently large $\eta$, the matrix $M(\xi,\omega,\eta)$ is similar to the following block-diagonal matrix 
\begin{equation}\label{tem4.9}
\eta\left({\begin{array}{*{20}c}
  \vspace{1.5mm}P_1(\sigma) & 0 \\
            0 & P_2(\sigma)
  \end{array}}\right)\quad \text{with}~~\sigma = 1/\eta.	
\end{equation}
Here $P_1(\sigma)$ and $P_2(\sigma)$ are analytic at $\sigma=0$ and 
\begin{align*}
P_1(\sigma)=&\left({\begin{array}{*{20}c}
  \vspace{1.5mm}0_{n-r-n_1^o} & 0 & 0 \\
  \vspace{1.5mm} 0 & 0_{n_1^o} & M_3 \\
                0  & 0 & 0_{n_1^o}  
  \end{array}}\right)+\sigma \left({\begin{array}{*{20}c}
  \vspace{1.5mm}M_1(\xi,\omega) & \star & \star \\
  \vspace{1.5mm}\star & \star & \star \\
                0  &  M_2(\xi,\omega) & \star  
  \end{array}}\right)+O(\sigma^2),\\[2mm]
 P_2(\sigma)=&M_4+O(\sigma).
\end{align*}
\end{lemma}

\begin{proof}
It suffices to show that the matrix $\bar{M}$ in \eqref{4.1} is similar to the matrix in \eqref{tem4.9}. To do this, we rewrite 
$$
\bar{M}\equiv \eta(\bar{M}_0+\sigma \bar{M}_1)
$$
with $\sigma=1/\eta$ and
\begin{align*}
\bar{M}_0= \left({\begin{array}{*{20}c}
  \vspace{1.5mm}0  & 0 & 0 & 0\\
  \vspace{1.5mm} 0 & 0  & M_3  & 0\\
  \vspace{1.5mm} 0  &  0 & 0  &0\\
  0 & 0 & 0 &  M_4
  \end{array}}\right),\quad
  \bar{M}_1=&\left({\begin{array}{*{20}c}
  \vspace{1.5mm}M_1(\xi,\omega) & \star & \star & \quad\star\\
  \vspace{1.5mm} \star & \star & \star  & \quad\star\\
  \vspace{1.5mm} 0  &  M_2(\xi,\omega) & \star &\quad\star\\
  \star & \star & \star &  \quad\star
  \end{array}}\right).
\end{align*}
Notice that the $(r-n_1^o)\times (r-n_1^o)$-matrix $M_4=(\tilde{K}^*A\tilde{K})^{-1}\tilde{K}^*S\tilde{K}$ is invertible. 

According to the classical perturbation theory of matrices \cite{Ka}, there exists an invertible matrix $T_1(\sigma)$, defined in a neighborhood of $\sigma=0$, such that 
\begin{equation}\label{tem4.10}
  T_1^{-1}(\sigma)(\bar{M}_0+\sigma \bar{M}_1)T_1(\sigma)=\left({\begin{array}{*{20}c}
  \vspace{1.5mm}P_1(\sigma) & 0 \\
            0 & P_2(\sigma)
  \end{array}}\right)
\end{equation}
holds for all small $\sigma$, $T_1(\sigma)$ and $T_1^{-1}(\sigma)$ are analytic at $\sigma=0$. Here $P_1(\sigma)$ and $P_2(\sigma)$ are the respective $(n-r+n_1^o)\times (n-r+n_1^o)$ and $(r-n_1^o)\times (r-n_1^o)$ matrices and $P_1(0)$ has $0$ as its eigenvalue with multiplicity $(n-r+n_1^o)$. 
Since $P_2(0)$ is invertible, the relation \eqref{tem4.10} implies that $T_1(0)$ is block-diagonal. Without loss of generality, we assume $T_1(0)=I_n$ and thereby obtain 
$$
P_2(\sigma)=M_4+O(\sigma)=(\tilde{K}^*A\tilde{K})^{-1}\tilde{K}^*S\tilde{K}+O(\sigma).
$$ 

Moreover, since $T_1(\sigma)$ and $T_1^{-1}(\sigma)$ are analytic, 
we can write $T_1(\sigma)=I_n+\sigma T_{p}+O(\sigma^2)$ and $T_1^{-1}(\sigma)=I_n+\sigma T_{r}+O(\sigma^2).$ Thus we compute 
\begin{align*}
&T_1^{-1}(\sigma)(\bar{M}_0+\sigma \bar{M}_1)T_1(\sigma) \\[2mm]
=&\left[I_n +\sigma T_{r}+O(\sigma^2)\right]\left[\bar{M}_0+\sigma \bar{M}_1\right]\left[I_n +\sigma T_{p}+O(\sigma^2)\right] \\[2mm]
=&\bar{M}_0+\sigma (T_r \bar{M}_0  + \bar{M}_0 T_p +\bar{M}_1)+O(\sigma^2)\\[2mm]
=&\bar{M}_0+\sigma  \left({\begin{array}{*{20}c}
  \vspace{1.5mm}0 & 0 & \star & \star\\
  \vspace{1.5mm} 0 & 0 & \star & \star\\
  \vspace{1.5mm} 0  &  0 & \star  &\star \\
  0 & 0 & \star & \star
  \end{array}}\right)
  +  \sigma \left({\begin{array}{*{20}c}
  \vspace{1.5mm}0 & 0 & 0 & 0\\
  \vspace{1.5mm} \star & \star & \star & \star\\
  \vspace{1.5mm} 0  &  0 & 0  & 0 \\
  \star & \star & \star & \star
  \end{array}}\right) +\sigma \bar{M}_1+O(\sigma^2).
\end{align*}
The last equation uses the expression of $\bar{M}_0$. Therefore it follows from \eqref{tem4.10} that
\begin{align*}
P_1(\sigma)=\left({\begin{array}{*{20}c}
  \vspace{1.5mm}0 & 0 & 0 \\
  \vspace{1.5mm} 0 & 0 & M_3 \\
                0  &  0 & 0  
  \end{array}}\right)+\sigma \left({\begin{array}{*{20}c}
  \vspace{1.5mm}M_1(\xi,\omega) & \star & \star \\
  \vspace{1.5mm}\star & \star & \star \\
                0  &  M_2(\xi,\omega) & \star  
  \end{array}}\right)+O(\sigma^2).
\end{align*}
This completes the proof.
\end{proof}

\subsection{Right-stable matrix}

Following the proofs above, we see from \eqref{tem4.7} and \eqref{tem4.10} that a right-stable matrix for $M(\xi,\omega,\eta)$ is 
\begin{equation}\label{tem4.11}
  R_M^S(\xi,\omega,\eta)=LL_0L_1T_1(1/\eta)\left({\begin{array}{*{20}c}
  \vspace{1.5mm} R_1^S(1/\eta) & \\
                  & R_2^S(1/\eta) 
  \end{array}}\right)
\end{equation}
with $\eta$ sufficiently large.
Here $R_1^S(1/\eta)$ and $R_2^S(1/\eta)$ are the right-stable matrices for $P_1(1/\eta)$ and $P_2(1/\eta)$, respectively. By the proof of Lemma \ref{lemma4.2}, $R_2^S(0)$ is the right-stable matrix of $P_2(0)=(\tilde{K}^*A\tilde{K})^{-1}\tilde{K}^*S\tilde{K}$. 
By Proposition \ref{prop3.3}, the matrix $P_2(0)$ has $(n^+-n_1^+-n_1^o)$ negative eigenvalues. Therefore, $R_2^S(1/\eta)$ is an $(r-n_1^o)\times(n^+-n_1^+-n_1^o)$-matrix when $\eta$ is sufficiently large.  

As to $R_1^S(1/\eta)$, we have

\begin{lemma}\label{lemma4.3}
Under Assumption \ref{Asp2.1}, the matrix $P_1(\sigma)$ defined in Lemma \ref{lemma4.2} is similar to the following block-diagonal matrix 
\begin{align*} 
\sqrt{\sigma} \left({\begin{array}{*{20}c}
  \vspace{1.5mm} P_{11}(\sqrt{\sigma}) & \\
                  & P_{12}(\sqrt{\sigma})
  \end{array}}\right).
\end{align*}
Here $P_{11}(\sqrt{\sigma})$ and $P_{12}(\sqrt{\sigma})$ are analytic with respect to $\sqrt{\sigma}$ and can be expressed as
\begin{align*}
P_{11}(\sqrt{\sigma})=&\sqrt{\sigma}M_1(\xi,\omega)+O(\sigma),\\[2mm]
P_{12}(\sqrt{\sigma})=&\left({\begin{array}{*{20}c}
  \vspace{1.5mm}0 & M_3\\
                   M_2(\xi,\omega)  & 0
  \end{array}}\right)+O(\sqrt{\sigma}). 
\end{align*}
\end{lemma}

\begin{proof}
Set 
\begin{equation}\label{tem4.12}
  \bar{L}_3(\sqrt{\sigma})=\left({\begin{array}{*{20}c}
  \vspace{1.5mm} I_{n-r-n_1^o} & & \\
  				  & I_{n_1^o} & \\ 
                  & & \sqrt{\sigma} I_{n_1^o} 
  \end{array}}\right).
\end{equation}
We use the expression of $P_1(\sigma)$ in Lemma \ref{lemma4.2} to compute
\begin{align}
&\bar{L}_3^{-1}(\sqrt{\sigma})P_1(\sigma)\bar{L}_3(\sqrt{\sigma})\nonumber\\[2mm]
=& ~\left({\begin{array}{*{20}c}
  \vspace{1.5mm} \sigma M_1(\xi,\omega) & \sigma \star & \sigma^{3/2} \star \\
  \vspace{1.5mm}\sigma \star & \sigma \star &  \sqrt{\sigma}M_3+ \sigma^{3/2} \star \\
                0  &   \sqrt{\sigma} M_2(\xi,\omega) & \sigma \star  
  \end{array}}\right)+O(\sigma^{3/2})\nonumber\\[2mm]
=& \sqrt{\sigma}\left[\left({\begin{array}{*{20}c}
  \vspace{1.5mm}0 & 0 & 0 \\
  \vspace{1.5mm}0 & 0  & M_3  \\
                  0  &  M_2(\xi,\omega) &  0 
  \end{array}}\right)+ 
  \sqrt{\sigma}\left({\begin{array}{*{20}c}
  \vspace{1.5mm}M_1(\xi,\omega) &  \star &  \star \\
  \vspace{1.5mm}  \star &   \star &  \star \\
                  0  & 0 &   \star  
  \end{array}}\right)+O(\sigma)\right].\label{tem4.13}
\end{align}
Recall from Lemma \ref{lemma4.1} that $M_2(\xi,\omega)$ and $M_3$ are invertible matices and thereby the matrix $\left({\begin{array}{*{20}c}
  \vspace{1.5mm}0 & M_3\\
  M_2(\xi,\omega) & 0 
  \end{array}}\right)$ is invertible. 

According to the classical perturbation theory of matrices \cite{Ka}, there exists an invertible matrix $T_2(\sqrt{\sigma})$ such that 
\begin{align}\label{tem4.14}
T_2^{-1}(\sqrt{\sigma})\bar{L}_3^{-1}(\sqrt{\sigma})P_1(\sigma)\bar{L}_3(\sqrt{\sigma})T_2(\sqrt{\sigma})
=\sqrt{\sigma} \left({\begin{array}{*{20}c}
  \vspace{1.5mm} P_{11}(\sqrt{\sigma}) & \\
                  & P_{12}(\sqrt{\sigma})
  \end{array}}\right)
\end{align}
holds for all sufficiently small $\sqrt{\sigma}$, $T_2(\sqrt{\sigma})$ and $T_2^{-1}(\sqrt{\sigma})$ are analytic at $\sqrt{\sigma}=0$, and $P_{11}(0)$ has $0$ as its eigenvalue with multiplicity $(n-r-n_1^o)$. Here $P_{11}(\sqrt{\sigma})$ and $P_{12}(\sqrt{\sigma})$ are $(n-r-n_1^o)\times (n-r-n_1^o)$ and $2n_1^o\times 2n_1^o$ matrices, respectively. Obviously, $P_{11}(0)=0$ and $P_{12}(0)$ is invertible. Thus $T_{2}(0)$ is a block-diagonal matrix and we can assume $T_{2}(0)=I_{n-r+n_1^o}$ without loss of generality. Consequently, we have 
$$
P_{12}(0)= \left({\begin{array}{*{20}c}
  \vspace{1.5mm}0 & M_3\\
  M_2(\xi,\omega) & 0 
  \end{array}}\right). 
$$
Moreover, a simple calculation using the analyticity of the matrices shows that 
\begin{equation*} 
  P_{11}(\sqrt{\sigma})= \sqrt{\sigma}M_1(\xi,\omega)+O(\sigma).
\end{equation*}
This completes the proof.
\end{proof}

With \eqref{tem4.13} and \eqref{tem4.14}, we can further compute $R_1^S(1/\eta)$ in \eqref{tem4.11} as 
\begin{equation}\label{tem4.15}
  R_1^S(1/\eta)=\bar{L}_3(\frac{1}{\sqrt{\eta}})T_2(\frac{1}{\sqrt{\eta}})\left({\begin{array}{*{20}c}
  \vspace{1.5mm}R_{11}^S(\frac{1}{\sqrt{\eta}}) &  \\
                    & R_{12}^S(\frac{1}{\sqrt{\eta}})
  \end{array}}\right)
\end{equation}
when $\eta$ is sufficiently large.
Here $R_{11}^S(\frac{1}{\sqrt{\eta}})$ and $R_{12}^S(\frac{1}{\sqrt{\eta}})$ are the respective right-stable matrices of $\sqrt{\eta}P_{11}(\frac{1}{\sqrt{\eta}})$ and $P_{12}(\frac{1}{\sqrt{\eta}})$. According to Lemma \ref{lemma4.3}, we find that $R_{11}^S(0)$ is a right-stable matrix of $M_1(\xi,\omega)$ and $R_{12}^S(0)$ is a right-stable matrix of $P_{12}(0)$. 

To have a more explicit expression of $R_{12}^S(0)$, we need the following fact.

\begin{prop}\label{lemma44}
If an invertible matrix has no negative eigenvalues, then it can be written as a square of a stable (or unstable) matrix. 

\end{prop}

\begin{proof}
In view of the Jordan canonical form, it suffices to consider a Jordan block $ \lambda I + J$ where $\lambda\neq 0$ and 
$$
J=\left({\begin{array}{*{20}c}
  \vspace{1.5mm}0 & 1 & 0 & \cdots & 0\\
  \vspace{1.5mm}0 & 0 & 1 & \cdots & 0\\
  \vspace{1.5mm}\vdots & \vdots & \ddots & \ddots & \vdots\\
  \vspace{1.5mm}0 & 0 & \cdots & 0 & 1\\
                0 & 0 & \cdots & 0 & 0\\
  \end{array}}\right).
$$ 
Considering Taylor's expansion $(1+x)^{1/2}=1+\frac{1}{2}x-\frac{1}{8}x^2+\frac{1}{16}x^3-\cdots$, we take
$$
S_0 = -\text{sign}(Re\sqrt{\lambda})\sqrt{\lambda} \left[I+\frac{1}{2\lambda}J-\frac{1}{8 \lambda^2}J^2+\frac{1}{16 \lambda^3}J^3+\cdots \right].
$$
Note that the summation is finite and $Re\sqrt{\lambda}\neq 0$. Clearly it holds that $S_0^2=\lambda I + J$ and this completes the proof.
\end{proof}

With this proposition, we prove

\begin{lemma}\label{lemma4.4}
The right-stable matrix of the $2n_1^o\times 2n_1^o$-matrix $P_{12}(0)=\left({\begin{array}{*{20}c}
  \vspace{1.5mm}0_{n_1^o} & M_3  \\
               M_2(\xi,\omega) &  0_{n_1^o}  
  \end{array}}\right)$ 
can be taken as the form
$$
R_{12}^S(0)=\left({\begin{array}{*{20}c}
  \vspace{1.5mm}I_{n_1^o} \\  
                \star
\end{array}}\right).
$$ 
\end{lemma}

\begin{proof}
If $M_3M_2(\xi,\eta)$ is invertible and has no negative eigenvalues, then it follows from Proposition \ref{lemma44} that $M_3M_2(\xi,\eta)=S_0^2$ with $S_0$ a stable matrix. Thus it is easy to see that $R_{12}^S(0)$ can be taken as 
$$
R_{12}^S(0) = \left({\begin{array}{*{20}c}
  \vspace{1.5mm}I_{n_1^o} \\  
                M_3^{-1}S_0
\end{array}}\right).
$$
Therefore, it suffices to show that $M_3M_2(\xi,\eta)$ is invertible and has no negative eigenvalues. By the definition in Lemma \ref{lemma4.1}, we have 
$$
M_3M_2(\xi,\eta)=S_K\big[\xi I + C_{00}(\omega)\big]
$$
with
$$
S_K=-(K^*K)^{-1}[K^*SK - (K^*S\tilde{K})(\tilde{K}^*S\tilde{K})^{-1}(\tilde{K}^*SK)](K^*K)^{-1} 
$$
positive definite,
which is similar to $S_K^{1/2}[\xi I+C_{00}(\omega)]S_K^{1/2}$. 
Let $\lambda$ be an eigenvalue of this matrix associated with eigenvector $r_\lambda$:
\begin{align*}
S_K^{1/2}[\xi I+C_{00}(\omega)]S_K^{1/2}r_\lambda=\lambda r_\lambda.
\end{align*}
Multiplying this relation with $r_\lambda^*$ on the left and taking the real part yield
$$
Re\xi (r_\lambda^*S_K r_\lambda)=Re\lambda |r_\lambda|^2.
$$
Thus $Re\lambda>0$ and hence it can not be zero or negative. 
This completes the proof.
\end{proof}

Next we take $\eta\rightarrow \infty$ in \eqref{tem4.15} to get
\begin{equation*} 
  R_1^S(0)= \bar{L}_3(0)T_2(0)\left({\begin{array}{*{20}c}
  \vspace{1.5mm}R_{11}^S(0) &  \\
                    & R_{12}^S(0)
  \end{array}}\right).
\end{equation*}
By the definition of $\bar{L}_3(1/\eta)$ in \eqref{tem4.12} and $T_2(0)=I$, we use Lemma \ref{lemma4.4} to obtain
$$
R_1^S(0)=\left({\begin{array}{*{20}c}I_{n-r-n_1^o} & & \\
    \vspace{1.5mm}  & I_{n_1^o} & \\
                     & & 0_{n_1^o}
  \end{array}}\right)
  \left({\begin{array}{*{20}c}
  \vspace{1.5mm}R_{11}^S(0) & 0 \\
  \vspace{1.5mm}0  &  I_{n_1^o}\\
                  0  & \star
  \end{array}}\right)=\left({\begin{array}{*{20}c}
  \vspace{1.5mm}R_{11}^S(0) & 0 \\
  \vspace{1.5mm}0  &  I_{n_1^o}\\
                  0  & 0_{n_1^o}
  \end{array}}\right).
$$
Thus we let $\eta \rightarrow \infty$ in the expression \eqref{tem4.11} and have

\begin{equation*} 
R_M^S \equiv R_M^S(\xi,\omega,\infty) 
= LL_0L_1\left({\begin{array}{*{20}c}
  \vspace{1.5mm}R_{11}^S(0) & 0 & 0\\
  \vspace{1.5mm}0 & I_{n_1^o} &0\\
  \vspace{1.5mm}0 & 0_{n_1^o} &0\\
                  0& 0& R_2^S(0) 
  \end{array}}\right).
\end{equation*}
Recall the expression of $LL_0L_1$ in 
\eqref{tem4.8}. We compute
\begin{align}
R_M^S 
 = \left({\begin{array}{*{20}c}
  \vspace{1.5mm}\big[P_1-P_0[\xi I_{n_1^o} +C_{00}(\omega)]^{-1}C_{01}(\omega)\big]R_{11}^S & P_0 & ~NR_2^S \\[2mm]
  0 & 0 & ~\tilde{K}R_2^S 
  \end{array}}\right)	\label{tem4.16}
\end{align}
with $R_{11}^S=R_{11}^S(0)$ and $R_2^S=R_2^S(0)$. 
Notice that this is a full-rank matrix. 

\subsection{Main results}

Now we take $\eta\rightarrow \infty$ in the Generalized Kreiss Condition to get 
\begin{equation*} 
  |\det\{BR_M^S \}|\geq c_K \sqrt{\det\{R_M^{S*} R_M^S \}}.
\end{equation*}
Because the matrix $R_M^S$ is of full-rank, the above inequality implies that
\begin{equation}\label{tem4.17}
 	BR_M^S=\left( 
   B_u\big[P_1-P_0[\xi I_{n_1^o} +C_{00}(\omega)]^{-1}C_{01}(\omega)\big]R_{11}^S, \ B_uP_0, \ B_uNR_2^S + B_v \tilde{K}R_2^S  
   \right)
\end{equation}
is invertible. 

With these preparations, we establish the following main result.

\begin{theorem}\label{thm4.5}
Under the assumptions in Section \ref{section2}, there exists a full-rank $n_1^+\times n^+$-matrix $B_o$ such that the relation
\begin{equation}\label{tem4.18}
	B_oB_u\bar{u}(t,0,\hat{x})=B_ob(t,\hat{x}),
\end{equation}
as a boundary condition for the equilibrium system, satisfies $B_oB_uP_0=0$ and the Uniform Kreiss Condition \cite{MO} for characteristic IBVPs:  
\begin{equation*}
  |\det\{B_oB_uP_1R_{11}^S\}| \geq c_K \sqrt{\det\{R_{11}^{S*}R_{11}^S \}}
\end{equation*}
with $c_K$ a positive constant and $R_{11}^S$ a right-stable matrix of $M_1(\xi,\omega)$.
\end{theorem}

\begin{remark}
To see why $M_1(\xi,\omega)$ and $P_1$ appear here, we refer to \cite{MO} and reformulate the Uniform Kreiss Condition for characteristic IBVPs of symmetrizable hyperbolic systems.
For this purpose, we multiply the equilibrium system \eqref{3.2} with the orthonormal matrix 
$\left({\begin{array}{*{20}c}
                 P_1^*\\ 
                 P_0^*
  \end{array}}\right)$ from the left to obtain
$$
\left({\begin{array}{*{20}c}
                 \vspace{1.5mm}P_1^*\bar{u}\\ 
                 P_0^*\bar{u}
  \end{array}}\right)_t+\left({\begin{array}{*{20}c}
                 \vspace{1.5mm}\Lambda_1 & 0\\ 
                 0 & 0
  \end{array}}\right)\left({\begin{array}{*{20}c}
                 \vspace{1.5mm}P_1^*\bar{u}\\ 
                 P_0^*\bar{u}
  \end{array}}\right)_{x_1}+\sum_{j=2}^d\left({\begin{array}{*{20}c}
                 \vspace{1.5mm}P_1^*A_{j11}P_1 & P_1^*A_{j11}P_0 \\ 
                 P_0^*A_{j11}P_1 & P_0^*A_{j11}P_0
  \end{array}}\right)\left({\begin{array}{*{20}c}
                 \vspace{1.5mm}P_1^*\bar{u}\\ 
                 P_0^*\bar{u}
  \end{array}}\right)_{x_j}=0.
$$
Set $ V_k= P_k^*\bar{u} ~(k=0,1)$. 
The corresponding eigenvalue problem of the last system is 
$$
\xi\left({\begin{array}{*{20}c}
                 \vspace{1.5mm}\hat{V}_1\\ 
                 \hat{V}_0
  \end{array}}\right) +\left({\begin{array}{*{20}c}
                 \vspace{1.5mm}\Lambda_1 & 0\\ 
                 0 & 0
  \end{array}}\right)\left({\begin{array}{*{20}c}
                 \vspace{1.5mm}\hat{V}_1\\ 
                 \hat{V}_0
  \end{array}}\right)_{x_1}+ \left({\begin{array}{*{20}c}
                 \vspace{1.5mm}C_{11}(\omega) & C_{10}(\omega) \\ 
                 C_{01}(\omega) & C_{00}(\omega)
  \end{array}}\right)\left({\begin{array}{*{20}c}
                 \vspace{1.5mm}\hat{V}_1\\ 
                 \hat{V}_0
  \end{array}}\right)=0, 
$$
where $C_{kl}(\omega) = P_k^*C(\omega)P_l$ and $C(\omega)= i \sum_{j=2}^d\omega_j A_{j11}$ as defined in the beginning of this section.
From this eigenvalue problem, it is easy to see that $\hat{V}_0=-[\xi I+C_{00}(\omega)]^{-1}C_{01}(\omega)\hat{V}_1$ and 
$$
\hat{V}_{1x}=-\Lambda_1^{-1}\bigg(\big[\xi I +C_{11}(\omega)\big]-C_{10}(\omega)\big[\xi I + C_{00}(\omega)\big]^{-1}C_{01}(\omega) \bigg)\hat{V}_{1}=M_1(\xi,\omega)\hat{V}_{1}. 
$$
Therefore, a boundary condition of the form $\tilde{B}\bar{u}(t,0,\hat{x})=\tilde{b}(t,\hat{x})$ for the equilibrium system can be rewritten as
$\tilde{B}P_1V_1(t,0,\hat{x})+\tilde{B}P_0V_0(t,0,\hat{x})=\tilde{b}(t,\hat{x})$ and  
should satisfy the Uniform Kreiss Condition $|\det\{\tilde{B}P_1R_{11}^S\}|\geq c_K\sqrt{\det\{R_{11}^{S*}R_{11}^S \}}$. In addition, it was shown in \cite{MO} that $\tilde{B}P_0=0$ is also necessary for the well-posedness.
Note that $\tilde{B}=B_oB_u$ and $\tilde{b}=B_ob$ for the boundary condition \eqref{tem4.18}.
\end{remark}

\begin{remark}
Surprisingly, the GKC implies that the reduced boundary condition \eqref{tem4.18} fulfills the necessary requirement $B_oB_uP_0=0$ proposed in \cite{MO}. 
\end{remark}

\textbf{Proof of Theorem 4.5:}
For convenience, we rename the block matrices of $BR_M^S$ in \eqref{tem4.17} as
\begin{align}
BR_M^S=&(Y_1,Y_2,Y_3)\nonumber\\
\equiv&\left( 
   B_u\big[P_1-P_0[\xi I_{n_1^o} +C_{00}(\omega)]^{-1}C_{01}(\omega)\big]R_{11}^S , \ B_uP_0, \ B_uNR_2^S + B_v \tilde{K}R_2^S  
   \right),  \label{tem4.19}
\end{align}
Since $BR_M^S$ is invertible, $(Y_2,Y_3)$ is an $n^+\times(n^+-n_1^+)$ full-rank matrix. Thus there exists a full-rank $n_1^+\times n^+$-matrix $B_o$ such that 
\begin{equation} \label{tem4.20}
    B_o(Y_2, Y_3)=0,
\end{equation}
which implies that $B_oB_uP_0=B_oY_2=0$. Recall that $R_2^S$ is a right-stable matrix of $P_2(0)=(\tilde{K}^*A\tilde{K})^{-1}\tilde{K}^*S\tilde{K}$, which is independent of parameters $\xi$ and $\omega$. Thus, $R_2^S$ and thereby $B_o$ are independent of the parameters.

With $B_o$ given thus, let $\tilde{B}_o$ be an $(n^+-n_1^+)\times n^+$-matrix such that 
$\left({\begin{array}{*{20}c}
                  B_o\\ 
                  \tilde{B}_o
  \end{array}}\right)$ is invertible. By \eqref{tem4.20}, we have
\begin{align*}
\left({\begin{array}{*{20}c}
    \vspace{1.5mm}B_o\\ 
      \tilde{B}_o
\end{array}}\right)BR_M^S 
=\left({\begin{array}{*{20}c}
    \vspace{1.5mm}B_oY_1 & 0 & 0\\[2mm] 
      \tilde{B}_oY_1 & \tilde{B}_oY_2 & \tilde{B}_oY_3
\end{array}}\right).
\end{align*}
Then the $(n^+-n_1^+)\times(n^+-n_1^+)$ matrix $(\tilde{B}_oY_2, ~\tilde{B}_oY_3)$ is invertible and 
\begin{align*} 
\det\{B_oY_1\}=\det\{(\tilde{B}_oY_2, ~\tilde{B}_oY_3)\}^{-1}\det\left\{\left({\begin{array}{*{20}c}
    \vspace{1.5mm}B_o\\ 
      \tilde{B}_o
\end{array}}\right)\right\}\det\{BR_M^S\}.
\end{align*}
Since $B_oB_uP_0=0$, it follows that $B_oY_1 = B_oB_uP_1R_{11}^S$.
Thus we deduce from the GKC that 
\begin{align}
	|\det\{B_oB_uP_1R_{11}^{S}\}|=&\det\{(\tilde{B}_oY_2, ~\tilde{B}_oY_3)\}^{-1}\det\left\{\left({\begin{array}{*{20}c}
    \vspace{1.5mm}B_o\\ 
      \tilde{B}_o
\end{array}}\right)\right\}\det\{BR_M^S\} \nonumber \\
 \geq& c_K\det\{(\tilde{B}_oY_2, ~\tilde{B}_oY_3)\}^{-1}\det\left\{\left({\begin{array}{*{20}c}
    \vspace{1.5mm}B_o\\ 
      \tilde{B}_o
\end{array}}\right)\right\}\sqrt{\det\{R_{M}^{S*} R_M^S \}}\label{tem4.21}\\
\equiv& \tilde{c}_1 \sqrt{\det\{R_{M}^{S*} R_M^S \}}.\nonumber
\end{align}

Now we rename the block matrices of $R_M^S$ in \eqref{tem4.16} as 
$$
R_M^S=\left({\begin{array}{*{20}c}
  \vspace{1.5mm}\bigg(\big[P_1-P_0[\xi I_{n_1^o} +C_{00}(\omega)]^{-1}C_{01}(\omega)\big]R_{11}^S, ~P_0\bigg) & ~NR_2^S \\[2mm]
  0  & ~\tilde{K}R_2^S 
  \end{array}}\right)\equiv \left({\begin{array}{*{20}c}
   \vspace{1.5mm} X_1 & X_2\\ 
                       0 & X_3
  \end{array}}\right).
$$
A simple computation shows 
\begin{align*}
R_M^{S*}R_M^S
=&\left({\begin{array}{*{20}c}
    \vspace{1.5mm}X_1^*X_1 &X_1^*X_2\\ 
      X_2^*X_1 & X_2^*X_2+X_3^*X_3
\end{array}}\right)\\[2mm]
=& \left({\begin{array}{*{20}c}
    \vspace{1.5mm}I  & 0\\ 
      X_2^*X_1(X_1^*X_1)^{-1} & I 
\end{array}}\right)\left({\begin{array}{*{20}c}
    \vspace{1.5mm}X_1^*X_1 & 0\\ 
      0 & \tilde{X}
\end{array}}\right)\left({\begin{array}{*{20}c}
    \vspace{1.5mm}I & (X_1^*X_1)^{-1}X_1^*X_2\\ 
      0 & I 
\end{array}}\right)
\end{align*}
with $\tilde{X}=X_2^*[I -X_1(X_1^*X_1)^{-1}X_1^*]X_2+X_3^*X_3$. 
From this it follows that 
\begin{equation}\label{tem4.22}
\det\{R_M^{S*}R_M^S \}=\det\{X_1^*X_1\} \det\{\tilde{X}\}.
\end{equation}
Since $I \geq X_1(X_1^*X_1)^{-1}X_1^*$ and $X_3^*X_3\geq 0$, we have
 $\det\{\tilde{X}\}\geq \det\{X_3^*X_3\}$ by Lemma 3.3 in \cite{Y2}. 
 Thus we deduce from \eqref{tem4.21} and \eqref{tem4.22} that
\begin{align}\label{tem4.23}
|\det\{B_oB_uP_1R_{11}^S\}|\geq&~\tilde{c}_1 \sqrt{\det\{\tilde{X}\}}\sqrt{\det\{X_1^*X_1\}}\geq 
\tilde{c}_1 \sqrt{\det\{X_3^{*}X_3\}}\sqrt{\det\{X_1^*X_1\}}.
\end{align}

Note that $X_3$ is of full-column-rank and independent of the parameters. Thus $\sqrt{\det\{X_3^{*}X_3\}}$ is a positive number independent of the parameters.
Moreover, let $\hat{C}=[\xi I_{n_1^o} +C_{00}(\omega)]^{-1}C_{01}(\omega)$. 
Then we have $X_1=\Big((P_1-P_0\hat{C})R_{11}^S, ~P_0\Big)$ and thereby
$$
X_1^*X_1=
\left({\begin{array}{*{20}c}
    \vspace{1.5mm}R_{11}^{S*}[I+\hat{C}^*\hat{C}]R_{11}^S & -R_{11}^{S*}\hat{C}^*\\ 
      -\hat{C}R_{11}^S & I 
\end{array}}\right)= 
\left({\begin{array}{*{20}c}
    \vspace{1.5mm}I & -R_{11}^{S*}\hat{C}^*\\ 
      0 & I 
\end{array}}\right)
\left({\begin{array}{*{20}c}
    \vspace{1.5mm}R_{11}^{S*} R_{11}^S & 0\\ 
      0 & I 
\end{array}}\right)
\left({\begin{array}{*{20}c}
    \vspace{1.5mm}I & 0\\ 
      -\hat{C}R_{11}^S & I 
\end{array}}\right)
$$
Therefore we have
$$
\sqrt{\det\{X_1^*X_1\}}= \sqrt{\det\{R_{11}^{S*}R_{11}^S\}}.
$$
This together with the inequality \eqref{tem4.23} completes the proof.

Furthermore, we prove 
\begin{lemma}\label{lemma4.6}
Under the GKC, the square matrix defined in \eqref{tem3.31}
is invertible. 
\end{lemma}

\begin{proof}
Taking $\xi=1$ and $\omega=0$ in the expression of \eqref{tem4.19}, we have
$$
BR_M^S(1,0,\infty)=\left( 
   B_uP_1R_{11}^S , \ B_uP_0, \ B_uNR_2^S + B_v \tilde{K}R_2^S  
   \right)
$$
and $R_{11}^S$ is the right-stable matrix of $M_1(1,0)=-\Lambda_1^{-1}$. Thus we can take $P_1^U= P_1R_{11}^S$ in the decomposition \eqref{tem3.28}.

Let $R_2^U$ and $R_2^S$ be the right-unstable and right-stable matrices of the invertible matrix $(\tilde{K}^*A\tilde{K})^{-1}\tilde{K}^*S\tilde{K}$. Then $(R_2^S, R_2^U)$ and $L_2^UR_2^U$ are invertible, $L_2^UR_2^S=0$, and $BR_M^S(1,0,\infty)$ is invertible due to the GKC. Hence the lemma follows from the following simple computation
\begin{align*}
&\left({\begin{array}{*{20}c}
  \vspace{1.5mm}B_uP_1^U & B_uP_0 & B_uN+B_v\tilde{K} \\ 
  0 & 0 & L_2^U
  \end{array}}\right)\left({\begin{array}{*{20}c}
  \vspace{1.5mm}I_{n_1^+} & 0 & 0 & 0\\ 
  \vspace{1.5mm}0 & I_{n_1^o} & 0 & 0\\ 
  0 & 0 & R_2^S & R_2^U
  \end{array}}\right) \\[3mm]
=&\left({\begin{array}{*{20}c}
  \vspace{1.5mm}B_uP_1^U & B_uP_0 & (B_uN + B_v \tilde{K})R_2^S & (B_uN + B_v \tilde{K})R_2^U \\[1mm]
  0 & 0 & 0 & L_2^UR_2^U
  \end{array}}\right)\\[3mm]
=&\left({\begin{array}{*{20}c}
  \vspace{1.5mm}BR_M^S(1,0,\infty)  & (B_uN + B_v \tilde{K})R_2^U \\
  0 & L_2^UR_2^U
  \end{array}}\right).
\end{align*}
\end{proof}

Finally, we obtain the boundary conditions for \eqref{3.2}, \eqref{tem3.14} and the initial condition for \eqref{tem3.8} from \eqref{tem3.27}.
First of all, we see from \eqref{tem3.30} that $w(t,0,\hat{x})$ can be expressed as $w(t,0,\hat{x})=R_2^Sw^S$ with $w^S$ an $(n^+-n_1^+-n_1^o)$-vector.
Thus \eqref{tem3.27} can be rewritten as 
\begin{equation}\label{tem4.24}
 B_u\bar{u}(t,0,\hat{x}) + B_uP_0(P_0^*\mu_1)(t,0,\hat{x}) + (B_uN+B_v\tilde{K})R_2^Sw^S = b(t,\hat{x}).
\end{equation}
Multiplying this equation with $B_o$ from the left and using \eqref{tem4.20}, we arrive at the boundary condition \eqref{tem4.18}. Consequently, $\bar{u}$ can be uniquely determined by solving the equilibrium system \eqref{3.2} together with \eqref{tem4.18} and proper initial data.

Once $\bar{u}$ is solved, we turn to derive the initial condition for \eqref{tem3.8} and the boundary condition for \eqref{tem3.14}.
Taking $\xi = 1$ and $\omega=0$ in \eqref{tem4.19}, we see that 
$$
(B_uP_1R_{11}^S, ~B_uNR_2^S + B_v \tilde{K}R_2^S)  \quad\text{and}\quad  (B_uP_1R_{11}^S , \ B_uP_0)
$$ 
are the respective $n^+\times(n^+-n_1^o)$ and $n^+\times (n_1^++ n_1^o)$ full-rank matrices. Then there exist an $n_1^o\times n^+$-matrix $B_1$ and an $(n^+-n_1^+-n_1^o)\times n^+$-matrix $B_{2}$ such that 
\begin{equation*} 
  B_1(B_uP_1R_{11}^S, ~B_uNR_2^S + B_v \tilde{K}R_2^S)=0, \qquad B_2(B_uP_1R_{11}^S , \ B_uP_0)=0.
\end{equation*}
Multiplying \eqref{tem4.24} with $B_1$ and $B_2$ from the left yields
\begin{equation}\label{tem4.25}
(B_1B_uP_0)(P_0^*\mu_1)(t,0,\hat{x})=B_1b(t,\hat{x}) - B_1B_u\bar{u}(t,0,\hat{x})	
\end{equation}
and 
\begin{equation}\label{tem4.26}
B_2(B_uN+B_v\tilde{K})R_2^Sw^S= B_2b(t,\hat{x}) - B_2B_u\bar{u}(t,0,\hat{x}).	
\end{equation}
Like $B_oY_1$ in the proof of Theorem \ref{thm4.5}, $B_1B_uP_0\equiv B_1Y_2$ and $B_2(B_uN+B_v\tilde{K})R_2^S\equiv B_2Y_3$ can be shown to be invertible. Thus $P_0^*\mu_1(t,0,\hat{x})$ and $w^S$ can be solved from \eqref{tem4.25} and \eqref{tem4.26}. At last, $w(t,0,\hat{x})$ is obtained by the simple relation $w(t,0,\hat{x})=R_2^Sw^S$ above. Consequently, we get the initial condition for \eqref{tem3.8} due to Proposition \ref{prop3.3} and the boundary condition for \eqref{tem3.14}.
 


\section{Validity}\label{section5} 
In this section, we show that the relaxation limit of smooth solution $U^\epsilon$ to the IBVP \eqref{2.1} satisfies the boundary condition \eqref{tem4.18} as well as the 
equilibrium system \eqref{3.2}. This will be done by estimating $(U^\epsilon - U_\epsilon)$ with the decomposition method \cite{GKO} mainly based on the GKC. Here $U_\epsilon$ is the approximate solution constructed in Section \ref{section3}.

For simplicity, we make some assumptions on the initial and boundary data. In order to avoid the initial-layer effect, the initial value $U_0(x_1,\hat{x})$ for the relaxation system \eqref{2.1} is assumed to be in equilibrium:
\begin{equation}\label{5.1}
	U_0(x_1,\hat{x})=\left({\begin{array}{*{20}c}
  \vspace{1.5mm}u_0(x_1,\hat{x}) \\
                0
\end{array}}\right),
\end{equation}
where $u_0(x_1,\hat{x})$ represents the first $(n-r)$ components of $U_0(x_1,\hat{x})$.  
Moreover, we assume that the initial value $U_0(x_1,\hat{x})$ and the boundary condition in \eqref{2.1} are compatible:
\begin{equation}\label{5.2}
  BU_0(0,\hat{x})=b(0,\hat{x}),\qquad \text{for}~\hat{x}\in \mathbb{R}^{d-1}.
\end{equation} 
These imply
\begin{equation}\label{5.3}
  B_oB_uu_0(0,\hat{x})=B_ob(0,\hat{x}),
\end{equation}
meaning that the reduced boundary condition \eqref{tem4.18} for the equilibrium system \eqref{3.2}  is compatible with the initial value $u_0$. 

\subsection{Approximate solution}\label{subsection5.1}
As a preparation for the estimate in the next subsection, we complete the construction of the approximate solution in \eqref{3.1}:
$$
U_\epsilon(t,x_1,\hat{x})=\left({\begin{array}{*{20}c}
  \vspace{1.5mm}\bar{u}\\
  \bar{v}
\end{array}}\right)(t,x_1,\hat{x})+\left({\begin{array}{*{20}c}
  \vspace{1.5mm}\mu_0 \\
  \nu_0
\end{array}}\right)(t,\frac{x_1}{\epsilon},\hat{x})+\left({\begin{array}{*{20}c}
  \vspace{1.5mm}\mu_1 \\
  \nu_1
\end{array}}\right)(t,\frac{x_1}{\sqrt{\epsilon}},\hat{x})+\sqrt{\epsilon}\left({\begin{array}{*{20}c}
  \vspace{1.5mm}\mu_2 \\
  \nu_2
\end{array}}\right)(t,\frac{x_1}{\sqrt{\epsilon}},\hat{x}).
$$
Recall from Subsection \ref{subsection3.1} that 
\begin{equation}\label{5.4}
	\bar{v}=0,\quad \nu_1=0,\quad P_1^*\mu_1=0 
\end{equation}
and the coefficients $\bar{u}$, $(\mu_0, \nu_0)$ and $P_0^*\mu_1$ satisfy the equations \eqref{3.2}, \eqref{tem3.8} and \eqref{tem3.14} respectively. Furthermore, the boundary conditions for \eqref{3.2},\eqref{tem3.14} and the initial condition for \eqref{tem3.8} are given in \eqref{tem4.18}, \eqref{tem4.25} and \eqref{tem4.26}, respectively. 
It remains to specify initial conditions for the partial differential equations \eqref{3.2} and \eqref{tem3.14}. 
By \eqref{5.4}, the initial value $U_\epsilon(0,x_1,\hat{x})$ is
$$
\left({\begin{array}{*{20}c}
  \vspace{1.5mm}\bar{u}\\
  0
\end{array}}\right)(0,x_1,\hat{x})+\left({\begin{array}{*{20}c}
  \vspace{1.5mm}\mu_0 \\
  \nu_0
\end{array}}\right)(0,\frac{x_1}{\epsilon},\hat{x})+\left({\begin{array}{*{20}c}
  \vspace{1.5mm}P_0(P_0^*\mu_1) \\
  0
\end{array}}\right)(0,\frac{x_1}{\sqrt{\epsilon}},\hat{x})+\sqrt{\epsilon}\left({\begin{array}{*{20}c}
  \vspace{1.5mm}\mu_2 \\
  \nu_2
\end{array}}\right)(0,\frac{x_1}{\sqrt{\epsilon}},\hat{x}).
$$

For $U_\epsilon$ to be a good approximation of the exact solution $U^\epsilon$, its initial value should be close to  
$$
U^\epsilon(0,x_1,\hat{x})=U_0(x_1,\hat{x})=\left({\begin{array}{*{20}c}
  \vspace{1.5mm}u_0(x_1,\hat{x}) \\
  0
\end{array}}\right). 
$$  
Thus we take
$$
\bar{u}(0,x_1,\hat{x})=u_0(x_1,\hat{x}),\qquad P_0^*\mu_1(0,\frac{x_1}{\sqrt{\epsilon}},\hat{x}) \equiv 0.
$$
With these, it follows from \eqref{tem4.26} and \eqref{5.2} that $w(0,0,\hat{x})=0$. By Proposition \ref{prop3.3}, we know that $(\mu_0,\nu_0)|_{t=0}\equiv 0$.
Moreover, by \eqref{tem3.10} and \eqref{tem3.15}, we know that $\nu_2(0,z,\hat{x})\equiv0$ and $P_1^*\mu_2(0,z,\hat{x})\equiv0$.
Note that $\mu_2$ is not uniquely determined by the equation \eqref{tem3.15}. Here, for simplicity, we may take $\mu_2(0,z,\hat{x})\equiv 0$. In conclusion, the initial value for $U_\epsilon$ can be taken as
$$
U_\epsilon(0,x_1,\hat{x})=U_0(x_1,\hat{x})=U^\epsilon(0,x_1,\hat{x}).
$$

Combining the initial condition $\bar{u}(0,x_1,\hat{x})=u_0(x_1,\hat{x})$ with the equilibrium system \eqref{3.2} and the reduced boundary condition \eqref{tem4.18}, we have the following IBVP for $\bar{u}$:
\begin{eqnarray*} 
\left\{{\begin{array}{*{20}l}
  \bar{u}_{t}+A_{11}\bar{u}_{x_1}+\sum_{j=2}^dA_{j11}\bar{u}_{x_j}=0,\\[2mm]
  B_oB_u\bar{u}(t,0,\hat{x})=B_ob(t,\hat{x}),\\[2mm]
  \bar{u}(0,x_1,\hat{x})=u_0(x_1,\hat{x}).
\end{array}}\right.
\end{eqnarray*} 
Having $\bar{u}$, we can get the initial value $w(t,0,\hat{x})=R_2^Sw^S$ from \eqref{tem4.26}.
Then we refer to Proposition \ref{prop3.3} and determine $(\mu_0, \nu_0)$. 
On the other hand, we know from \eqref{tem3.14}, \eqref{tem4.25} and the initial condition $P_0^*\mu_1(0,z,\hat{x})=0$ that $P_0^*\mu_1$ satisfies 
\begin{eqnarray*} 
\left\{{\begin{array}{*{20}l}
 \partial_t(P_0^*\mu_1)+\left[P_0^*A_{12}S^{-1}A_{12}^*P_0\right]\partial_{zz}(P_0^*\mu_1)+\sum\limits_{j=2}^d (P_0^*A_{j11}P_0) \partial_{x_j} (P_0^*\mu_1)=P_0^*G,\\[4mm]
 (B_1B_uP_0)(P_0^*\mu_1)(t,0,\hat{x})=B_1b(t,\hat{x}) - B_1B_u\bar{u}(t,0,\hat{x}),\\[3mm]
  (P_0^*\mu_1)(0,z,\hat{x})=0.
\end{array}}\right.
\end{eqnarray*}
At last, we determine $\mu_2$ and $\nu_2$ according to \eqref{tem3.15} and \eqref{tem3.10}.
In this way, the approximate solution $U_\epsilon$ is constructed. 


For the error estimate in the next subsection, we make the following regularity assumption 
\begin{assumption}\label{asp5.1}
The expansion coefficients in \eqref{3.1} satisfy
\begin{eqnarray*} 
\left\{{\begin{array}{*{20}l}
\bar{u}, ~P_0^*\mu_1\in L^2([0,T]\times H^1(\mathbb{R}^+\times \mathbb{R}^{d-1})),\\[3mm]
(\mu_0,\nu_0), ~(\mu_2, \nu_2)\in H^1([0,T]\times \mathbb{R}^+ \times \mathbb{R}^{d-1}).
\end{array}}\right. 
\end{eqnarray*}
\end{assumption}
\subsection{Error estimate}\label{subsection5.2}
To estimate the difference 
$$
W:=U^\epsilon-U_\epsilon=\left({\begin{array}{*{20}c}
  \vspace{1.5mm}u^\epsilon \\
  v^\epsilon
\end{array}}\right)-\left({\begin{array}{*{20}c}
  \vspace{1.5mm}u_\epsilon \\
  v_\epsilon
\end{array}}\right),
$$
we firstly derive the equation of $W$. Recall the linear operator $\mathcal{L}$: 
$$
\mathcal{L}(U)=\partial_{t}U +\sum\limits_{j=1}^d A_j\partial_{x_j}U-QU/\epsilon 
$$
introduced in Section \ref{subsection3.1}.
Clearly, we have $\mathcal{L}(U^\epsilon)=0$ and
$$
\mathcal{L}(U_\epsilon)=\mathcal{L}_1+\mathcal{L}_2.
$$
Here 
\begin{align*}
\mathcal{L}_1=&\partial_{t}\left({\begin{array}{*{20}c}
  \vspace{1.5mm}\bar{u}\\
  \bar{v}
  \end{array}}\right) +
\sum\limits_{j=1}^d A_j
\partial_{x_j}\left({\begin{array}{*{20}c}
  \vspace{1.5mm}\bar{u} \\
  \bar{v}
  \end{array}}\right) 
  -\frac{1}{\epsilon}
\left({\begin{array}{*{20}c}
  \vspace{1.5mm}0 & 0 \\
  0 & S
  \end{array}}\right)
\left({\begin{array}{*{20}c}
  \vspace{1.5mm}\bar{u} \\
  \bar{v}
  \end{array}}\right),\\[2mm]
\mathcal{L}_2=&\partial_{t}\left({\begin{array}{*{20}c}
  \vspace{1.5mm}\mu\\
  \nu
  \end{array}}\right) +
\sum\limits_{j=1}^d A_j
\partial_{x_j}\left({\begin{array}{*{20}c}
  \vspace{1.5mm}\mu\\
  \nu
  \end{array}}\right)-\frac{1}{\epsilon}
\left({\begin{array}{*{20}c}
  \vspace{1.5mm}0 & 0 \\
  0 & S
  \end{array}}\right)
\left({\begin{array}{*{20}c}
  \vspace{1.5mm}\mu\\
  \nu
  \end{array}}\right).
\end{align*}
By \eqref{3.2} and \eqref{3.3}, we have
\begin{align*} 
\mathcal{L}_1=\left({\begin{array}{*{20}c}
  \vspace{2mm}0 \\
  A_{12}^*\partial_{x_1}\bar{u}+\sum\limits_{j=2}^d A_{j12}^*\partial_{x_j}\bar{u}
  \end{array}}\right).
\end{align*}
Using \eqref{tem3.8}-\eqref{tem3.10}, we obtain
\begin{align*}
\mathcal{L}_2=&~\partial_t\left({\begin{array}{*{20}c}
  \vspace{1.5mm}\mu \\
  \nu 
\end{array}}\right) +\sum\limits_{j=2}^d \left({\begin{array}{*{20}c}
  \vspace{1.5mm}A_{j11} & A_{j12} \\
  A_{j12}^* & A_{j22}
  \end{array}}\right) 
\partial_{x_j}\left({\begin{array}{*{20}c}
  \vspace{1.5mm}\mu \\
  \nu 
\end{array}}\right)\\[2mm]
& + \left({\begin{array}{*{20}c}
  \vspace{1.5mm}A_{11} & A_{12} \\
  A_{12}^* & A_{22}
  \end{array}}\right)\bigg[\frac{1}{\epsilon}
~\partial_y\left({\begin{array}{*{20}c}
  \vspace{1.5mm}\mu_0 \\
  \nu_0
\end{array}}\right) +\frac{1}{\sqrt{\epsilon}}
~\partial_z\left({\begin{array}{*{20}c}
  \vspace{1.5mm}\mu_1 \\
  \nu_1
\end{array}}\right) + 
\partial_z\left({\begin{array}{*{20}c}
  \vspace{1.5mm}\mu_2 \\
  \nu_2
\end{array}}\right) \bigg]\\[2mm]
&-\frac{1}{\epsilon}
\left({\begin{array}{*{20}c}
  \vspace{1.5mm}0 & 0 \\
  0 & S
  \end{array}}\right)\bigg[ 
\left({\begin{array}{*{20}c}
  \vspace{1.5mm}\mu_0 \\
  \nu_0
\end{array}}\right)+ 
\left({\begin{array}{*{20}c}
  \vspace{1.5mm}\mu_1 \\
  \nu_1
\end{array}}\right)+\sqrt{\epsilon} 
\left({\begin{array}{*{20}c}
  \vspace{1.5mm}\mu_2 \\
  \nu_2
\end{array}}\right)\bigg]\\[2mm]
=&~\partial_t\left({\begin{array}{*{20}c}
  \vspace{1.5mm}\mu \\
  \nu 
\end{array}}\right)  +\sum\limits_{j=2}^d \left({\begin{array}{*{20}c}
  \vspace{1.5mm}A_{j11} & A_{j12} \\
  A_{j12}^* & A_{j22}
  \end{array}}\right) 
\partial_{x_j}\left({\begin{array}{*{20}c}
  \vspace{1.5mm}\mu \\
  \nu 
\end{array}}\right)+\left({\begin{array}{*{20}c}
  \vspace{1.5mm}A_{11} & A_{12} \\
  A_{12}^* & A_{22}
  \end{array}}\right)
\partial_z\left({\begin{array}{*{20}c}
  \vspace{1.5mm}\mu_2 \\
  \nu_2
\end{array}}\right).
\end{align*}
Furthermore, we deduce from \eqref{tem3.13} that
\begin{align*} 
\mathcal{L}_2=&\left({\begin{array}{*{20}c}
  \vspace{1mm}0 \\
  \partial_t\nu_{0}+ \sum\limits_{j=2}^d\Big[A_{j12}^*\partial_{x_j}(\mu_0+\mu_1)+A_{j22}\partial_{x_j}\nu_{0}\Big]+A_{12}^*\partial_z\mu_2+A_{22}\partial_z\nu_2
\end{array}}\right)\nonumber \\[2mm]
&+ \sqrt{\epsilon}
\left[
\partial_t\left({\begin{array}{*{20}c}
  \vspace{1.5mm}\mu_2 \\
  \nu_2
\end{array}}\right) +\sum\limits_{j=2}^d
\left({\begin{array}{*{20}c}
  \vspace{1.5mm}A_{j11} & A_{j12}\\
  A_{j12}^* & A_{j22}
  \end{array}}\right)
\partial_{x_j}\left({\begin{array}{*{20}c}
  \vspace{1.5mm}\mu_2 \\
  \nu_2
\end{array}}\right)
\right].
\end{align*}
In summary, we can write 
$$
\mathcal{L}(U_\epsilon)=\mathcal{L}_1+\mathcal{L}_2=\left({\begin{array}{*{20}c}
  \vspace{1.5mm}\sqrt{\epsilon}R_{1}\\
                           R_{2} 
  \end{array}}\right) 
$$
with $R_1$ and $R_2$ clearly defined.  
By the regularity Assumption \ref{asp5.1}, we can see that $R_1,R_2\in L^2([0,T]\times \mathbb{R}^+\times \mathbb{R}^{d-1})$. 
Consequently, the equation for $W$ is 
\begin{align}\label{tem5.5}
\partial_tW+
\sum\limits_{j=1}^d A_j
\partial_{x_j} W=\frac{1}{\epsilon}
\left({\begin{array}{*{20}c}
  \vspace{1.5mm}0 & 0 \\
  0 & S
  \end{array}}\right)W -\left({\begin{array}{*{20}c}
  \vspace{1.5mm}\sqrt{\epsilon}R_{1} \\
  R_{2}
\end{array}}\right).
\end{align}

Next we turn to the boundary and initial conditions for $W$. From the original boundary condition in \eqref{2.1} and \eqref{tem3.26}, 
it is easy to see that the boundary condition for $W$ is 
\begin{eqnarray}\label{tem5.6}
  BW(t,0,\hat{x})=\sqrt{\epsilon}\left({\begin{array}{*{20}c}
  \vspace{1.5mm}\mu_2 \\
  \nu_2
  \end{array}}\right)(t,0,\hat{x})\equiv g(t,\hat{x}).
\end{eqnarray} 
By Assumption \ref{asp5.1}, $\mu_2(t,0,\hat{x}),\nu_2(t,0,\hat{x})\in L^2([0,T]\times \mathbb{R}^{d-1})$ and thereby
\begin{equation*} 
	\|g\|_{L^2([0,T]\times \mathbb{R}^{d-1})}\leq  C\sqrt{\epsilon}.
\end{equation*} 
On the other hand, recall the discussion in Section \ref{subsection5.1} that $U_\epsilon(0,x_1,\hat{x})=U_0(x_1,\hat{x})=U^\epsilon(0,x_1,\hat{x})$ and thereby 
\begin{align}\label{5.9}
W(0,x_1,\hat{x}) \equiv 0.
\end{align}
Combining \eqref{tem5.5}-\eqref{5.9}, we have the following IBVP for $W$:
\begin{align*} 
\left\{
{\begin{array}{*{20}l}
  \vspace{1.5mm}\partial_tW+
\sum\limits_{j=1}^d A_j
\partial_{x_j} W=\dfrac{1}{\epsilon}
\left({\begin{array}{*{20}c}
  \vspace{1.5mm}0 & 0 \\
  0 & S
  \end{array}}\right)W -\left({\begin{array}{*{20}c}
  \vspace{1.5mm}\sqrt{\epsilon}R_{1} \\
  R_{2}
\end{array}}\right), \\[3mm]
BW (t,0,\hat{x})=g(t,\hat{x}),\\[2mm]
W(0,x_1,\hat{x})=0.
\end{array}}
\right.
\end{align*}

To estimate $W$, we follow \cite{GKO} and decompose $W=W_1+W_2$. Here $W_1$ satisfies
\begin{align}\label{5.10}
\left\{
{\begin{array}{*{20}l}
  \vspace{1.5mm}\partial_tW_1+
\sum\limits_{j=1}^d A_j
\partial_{x_j} W_1=\dfrac{1}{\epsilon}
\left({\begin{array}{*{20}c}
  \vspace{1.5mm}0 & 0 \\
  0 & S
  \end{array}}\right)W_1 -\left({\begin{array}{*{20}c}
  \vspace{1.5mm}\sqrt{\epsilon}R_{1} \\
  R_{2}
\end{array}}\right), \\[4mm]
P_+W_1(t,0,\hat{x})=0,\\[2mm]
W_1(0,x_1,\hat{x})=0,
\end{array}}
\right.
\end{align}
and $W_2$ satisfies
\begin{align}\label{5.11}
\left\{
{\begin{array}{*{20}l}
  \vspace{1.5mm}\partial_tW_2+
\sum\limits_{j=1}^d A_j
\partial_{x_j} W_2=\dfrac{1}{\epsilon}
\left({\begin{array}{*{20}c}
  \vspace{1.5mm}0 & 0 \\
  0 & S
  \end{array}}\right)W_2, \\[2mm]
BW_2(t,0,\hat{x})=g(t,\hat{x})-BW_1(t,0,\hat{x}),\\[2mm]
W_2(0,x_1,\hat{x})=0.
\end{array}}\qquad \qquad 
\right.
\end{align}
In \eqref{5.10}, $P_+$ is an $n^+\times n$-matrix consisting of the $n^+$ left-eigenvectors of $A_1$ associated with the positive eigenvalues. Because $A_1$ is symmetric, $P_+$ can
be the first $n^+$ rows of the following orthonormal matrix $P$ satisfying
$$
PA_1P^*=\left({\begin{array}{*{20}c}
  \vspace{1.5mm}\Lambda_+ &  \\
                            & \Lambda_-
\end{array}}\right).
$$
Here $\Lambda_+$ and $\Lambda_-$ are diagonal matrices whose entries are $n^+$ positive eigenvalues and $(n-n^+)$ negative eigenvalues of $A_1$. Corresponding to the partition $\text{diag}(\Lambda_+,\Lambda_-)$, we can write $P=\left({\begin{array}{*{20}c}
   P_+\\
   P_-
\end{array}}\right),$ where $P_-$ is an $(n-n^+)\times n$-matrix. 

For IBVPs \eqref{5.10} and \eqref{5.11}, we have the following conclusions.

\begin{lemma}\label{lemma5.1}
The IBVP \eqref{5.10} has an unique solution $W_1=W_1(t,x_1,\hat{x})$ satisfying
\begin{align}
&\max_{t\in[0,T]}\|W_1(t,\cdot,\cdot)\|_{L^2(\mathbb{R}^+\times \mathbb{R}^{d-1})}^2+ \|W_1|_{x_1=0}\|_{L^2([0,T]\times \mathbb{R}^{d-1})}^2 \nonumber \\[2mm]
\leq & C(T)\epsilon\bigg(\|R_{1}\|^2_{L^2([0,T]\times \mathbb{R}^+\times \mathbb{R}^{d-1})}
+ \|R_{2}\|^2_{L^2([0,T]\times \mathbb{R}^+\times \mathbb{R}^{d-1})}\bigg).\label{5.12}
\end{align}
Here $C(T)$ is a generic constant depending only on $T$.
\end{lemma}

\begin{lemma}\label{lemma5.2}
The IBVP \eqref{5.11} has an unique solution $W_2=W_2(t,x_1,\hat{x})$ satisfying 
\begin{align}
&\max_{t\in[0,T]}\|W_2(t,\cdot,\cdot)\|_{L^2(\mathbb{R}^+\times \mathbb{R}^{d-1})}^2+ \|W_2|_{x_1=0}\|_{L^2([0,T]\times \mathbb{R}^{d-1})}^2 \nonumber \\[2mm]
\leq & C(T)\bigg(\|W_1|_{x_1=0}\|_{L^2([0,T]\times \mathbb{R}^{d-1})}^2+\|g\|_{L^2([0,T]\times \mathbb{R}^{d-1})}^2\bigg).\label{5.13}
\end{align}
\end{lemma}

\noindent Recall that $\|(R_1,R_2)\|_{L^2([0,T]\times \mathbb{R}^+\times \mathbb{R}^{d-1})}$ is bounded and $\|g\|_{L^2([0,T]\times \mathbb{R}^{d-1})}\leq  C\sqrt{\epsilon}$. 
The above two lemmas immediately give

\begin{theorem}\label{thm5.3}
Under the assumptions in Section \ref{section2} and Assumption \ref{asp5.1}, there exists a constant $K>0$ such that the error estimate 
\begin{align*}
\|(U^\epsilon-U_\epsilon)(t,\cdot,\cdot)\|_{L^2(\mathbb{R}^+\times \mathbb{R}^{d-1})}\leq K \epsilon^{1/2} 
\end{align*}
holds for all time $t\in[0,T]$.
\end{theorem}

It remains to prove Lemma \ref{lemma5.1} and Lemma \ref{lemma5.2}, which is similar those in \cite{ZY}. For completeness, we present the details here. 

\begin{proof}(Lemma \ref{lemma5.1})
Clearly, the boundary condition $P_+W_1(0,t)=0$ satisfies the Uniform Kreiss Condition. Thus it follows from the existence theory in \cite{BS} that there exists an unique solution $W_1\in C([0,T];L^2(\mathbb{R}^+\times \mathbb{R}^{d-1}))$.

For the estimate \eqref{5.12}, we multiply \eqref{5.10} with $W_1^*$ from the left to obtain
$$
\frac{d}{dt}(W_1^*W_1)+\sum_{j=1}^d(W_1^*A_jW_1)_{x_j}= \dfrac{2}{\epsilon}
W_1^*\left({\begin{array}{*{20}c}
  \vspace{1.5mm}0 & 0 \\
  0 & S
  \end{array}}\right)W_1-2\sqrt{\epsilon}W_1^{I*}R_{1}-2W_1^{II*}R_{2}.
$$
Here $W_1^I$ represents the first $(n-r)$ components of $W_1$ and $W_1^{II}$ represents the other $r$ components. From the last equation, we use the negative-definiteness of $S$ and derive
\begin{align*}
\frac{d}{dt}(W_1^*W_1)+\sum_{j=1}^d(W_1^*A_jW_1)_{x_j}\leq & -\frac{c_0}{\epsilon}|W_1^{II}|^2+ 2|W_1^{II*}R_{2}| + 2\sqrt{\epsilon}|W_1^{I*}R_{1}|\\[2mm]
\leq &  -\frac{c_0}{2\epsilon}|W_1^{II}|^2 + \frac{2 \epsilon}{c_0}|R_{2}|^2 + \epsilon|R_{1}|^2 + |W_1|^2 
\end{align*}
with $c_0 > 0$ a constant. Integrating the last inequality over $(x_1,\hat{x})\in[0,+\infty)\times \mathbb{R}^{d-1}$ yields
\begin{align}
&~\frac{d}{dt}\|W_1(t,\cdot,\cdot)\|_{L^2(\mathbb{R}^+\times \mathbb{R}^{d-1})}^2 - \int_{\mathbb{R}^{d-1}}W_1^*(t,0,\hat{x}) A_1W_1(t,0,\hat{x}) d\hat{x}\nonumber \\
 \leq &~\frac{2 \epsilon}{c_0}\|R_{2}\|_{L^2(\mathbb{R}^+\times \mathbb{R}^{d-1})}^2 + \epsilon\|R_{1}\|_{L^2(\mathbb{R}^+\times \mathbb{R}^{d-1})}^2 + \|W_1\|_{L^2(\mathbb{R}^+\times \mathbb{R}^{d-1})}^2  .\label{5.14}
\end{align}
Since $A_1=P_+^*\Lambda_+P_+ + P_-^*\Lambda_-P_-$ and $P_+W_1(t,0,\hat{x})=0$, it follows that 
\begin{align}
-W_1^*(t,0,\hat{x})A_1W_1(t,0,\hat{x})=& -W_1^*(t,0,\hat{x})(P_+^*\Lambda_+P_+ + P_-^*\Lambda_-P_-)W_1(t,0,\hat{x})\nonumber\\[2mm]
                 =& -W_1^*(t,0,\hat{x}) P_-^*\Lambda_-P_- W_1(t,0,\hat{x}) \nonumber\\[2mm]
                 \geq & ~c_1W_1^*(t,0,\hat{x}) P_-^* P_- W_1(t,0,\hat{x}) \nonumber\\[2mm]
                 =& ~c_1|W_1(t,0,\hat{x})|^2\label{5.15}
\end{align}
with $c_1$ a positive constant.
Applying Gronwall's inequality to \eqref{5.14}, we have 
\begin{align*}
 \max_{t\in[0,T]}\|W_1(t,\cdot,\cdot)\|_{L^2(\mathbb{R}^+\times \mathbb{R}^{d-1})}^2
\leq & ~ Ce^{CT}\epsilon\bigg( \|R_{1}\|^2_{L^2([0,T] \times \mathbb{R}^+\times \mathbb{R}^{d-1})}+ \|R_{2}\|^2_{L^2([0,T] \times \mathbb{R}^+\times \mathbb{R}^{d-1})} \bigg). 
\end{align*}
At last, we use \eqref{5.15} and integrate \eqref{5.14} over $t\in[0,T]$ to obtain
\begin{align*}
\|W_1|_{x_1=0}\|_{L^2([0,T]\times \mathbb{R}^{d-1})}^2
\leq & ~C(T) \epsilon\bigg(\|R_{1}\|^2_{L^2([0,T] \times \mathbb{R}^+\times \mathbb{R}^{d-1})}+ \|R_{2}\|^2_{L^2([0,T] \times \mathbb{R}^+\times \mathbb{R}^{d-1})} \bigg). 
\end{align*}
This completes the proof of Lemma \ref{lemma5.1}.
\end{proof}

Finally, we present a proof of Lemma \ref{lemma5.2}.
\begin{proof}(Lemma \ref{lemma5.2})
By Remark 3.2 in \cite{Y2}, the Uniform Kreiss Condition is implied by the GKC. Thus the existence theory in \cite{BS} indicates that there exists an unique solution $W_2\in C([0,T];L^2(\mathbb{R}^+\times \mathbb{R}^{d-1}))$.

Next we adapt the method in \cite{GKO} to obtain the estimate \eqref{5.13}. Define the Fourier transform of the solution to \eqref{5.11} with respect to $\hat{x}$:
$$
\tilde{U}(t,x_1,\omega)=\int_{\mathbb{R}^{d-1}}e^{i \omega \hat{x}}U(t,x_1,\hat{x})d\hat{x}, \qquad \omega=(\omega_2,\omega_3,...,\omega_d)\in \mathbb{R}^{d-1}
$$
and the Laplace transform with respect to $t$:
$$
\hat{U}(\xi,x_1,\omega)=\int_0^{\infty}e^{-\xi t}\tilde{U}(t,x_1,\omega)dt,\qquad Re \xi>0. 
$$
Then we deduce from \eqref{5.11} that 
\begin{align}\label{5.16}
\left\{
{\begin{array}{*{20}l}
\vspace{2mm}\partial_{x_1}\hat{W_2} =A_1^{-1}\left(\eta Q - \xi I_n -  i \sum_{j=2}^d\omega_j A_j\right)\hat{W_2}\equiv M(\xi,\omega,\eta)\hat{W_2},\\[2mm]
\vspace{2mm} B\hat{W}_2(\xi,0,\omega)=\hat{g}(\xi,\omega)-B\hat{W}_1(\xi,0,\omega),\\[2mm]
\|\hat{W_2}(\xi,\cdot,\omega)\|_{L^2(\mathbb{R}^+)}<\infty \quad \text{for ~a.e.} ~~\xi,~\omega. 
\end{array}}
\right.
\end{align}
Here $\eta = 1/\epsilon$.
Let $R_M^S=R_M^S(\xi,\omega,\eta)$ and $R_M^U=R_M^U(\xi,\omega,\eta)$ be the respective right-stable and right-unstable matrices of $M=M(\xi,\omega,\eta)$:
\begin{align*}
MR_M^S=R_M^SM^S,\qquad\qquad 
MR_M^U=R_M^UM^U,
\end{align*}
where $M^S$ is a stable-matrix and $M^U$ is an unstable-matrix. 
In view of the Schur decomposition, we may choose $R_M^S$ and $R_M^U$ such that
$$
\left({\begin{array}{*{20}c}
  \vspace{1.5mm} R_M^{S*} \\
                 R_M^{U*}
\end{array}}\right)(R_M^S,\ R_M^U)=I_n.
$$
Then from \eqref{5.16} we obtain
\begin{align*}
\left({\begin{array}{*{20}c}
  \vspace{1.5mm} R_M^{S*} \\
                 R_M^{U*}
\end{array}}\right)\partial_{x_1}\hat{W}_{2}= 
  \left({\begin{array}{*{20}c}
  \vspace{1.5mm}M^S & \\
                    & M^U
\end{array}}\right)\left({\begin{array}{*{20}c}
  \vspace{1.5mm} R_M^{S*} \\
                 R_M^{U*}
\end{array}}\right)\hat{W}_2.
\end{align*}
Since $\|\hat{W}_2(\xi,\cdot,\omega)\|_{L^2(\mathbb{R}^+)}<\infty$ for a.e. $(\xi,\omega)$ and $M^U$ is an unstable-matrix, it must be
\begin{equation*}
   R_M^{U*}\hat{W}_2=0.
\end{equation*}

Thus the boundary condition in \eqref{5.16} becomes
\begin{align*}
B\hat{W}_2(\xi,0,\omega)= BR_M^{S}R_M^{S*}\hat{W}_2(\xi,0,\omega)=&\hat{g}(\xi,\omega)-B\hat{W}_1(\xi,0,\omega). 
\end{align*}
Since the matrix $(BR_M^S)^{-1}$ is uniformly bounded due to the GKC, we conclude that
\begin{align*}
\left|\hat{W}_2(\xi,0,\omega)\right|^2=& \left|R_M^{S}R_M^{S*}\hat{W}_2(\xi,0,\omega)\right|^2\\[2mm]
=&\left|R_M^{S}(BR_M^S)^{-1}\left[\hat{g}(\xi,\omega)-B\hat{W}_1(\xi,0,\omega)\right]\right|^2\\[2mm]
\leq & ~C\left(|\hat{g}(\xi,\omega)|^2+|\hat{W}_1(\xi,0,\omega)|^2\right).
\end{align*}
By Parseval's identity, the last inequality leads to 
\begin{align*} 
&\int_{\mathbb{R}^{d-1}}\int_0^{\infty}e^{-2t Re\xi }|W_2(t,0,\hat{x})|^2dtd\hat{x} \\[2mm]
\leq&  C\bigg(\int_{\mathbb{R}^{d-1}}\int_0^{\infty}e^{-2t Re\xi }\left|g(t,\hat{x})\right|^2dtd\hat{x}+\int_{\mathbb{R}^{d-1}}\int_0^{\infty}e^{-2t Re\xi }\left|W_1(t,0,\hat{x})\right|^2dtd\hat{x}\bigg)\\[2mm]
\leq& C\bigg(\int_{\mathbb{R}^{d-1}}\int_0^{\infty} \left|g(t,\hat{x})\right|^2dtd\hat{x}+\int_{\mathbb{R}^{d-1}}\int_0^{\infty} \left|W_1(t,0,\hat{x})\right|^2dtd\hat{x}\bigg)\quad\text{for}~~Re\xi>0.
\end{align*}
Because the right-hand side is independent of $Re\xi$, we have 
$$
\int_{\mathbb{R}^{d-1}}\int_0^{\infty} |W_2(t,0,\hat{x})|^2dtd\hat{x} \leq C\bigg(\int_{\mathbb{R}^{d-1}}\int_0^{\infty} \left|g(t,\hat{x})\right|^2dtd\hat{x}+\int_{\mathbb{R}^{d-1}}\int_0^{\infty} \left|W_1(t,0,\hat{x})\right|^2dtd\hat{x}\bigg).
$$
By using the trick from \cite{GKO}, the integral interval $[0,\infty)$ in the last inequality can be changed to $[0,T]$:
\begin{align}
\int_{\mathbb{R}^{d-1}}\int_0^{T} |W_2(t,0,\hat{x})|^2dtd\hat{x} 
\leq & ~C \bigg(\|W_1|_{x_1=0}\|_{L^2([0,T]\times \mathbb{R}^{d-1})}^2+\|g\|_{L^2([0,T]\times \mathbb{R}^{d-1})}^2\bigg). \label{5.17}
\end{align}

At last, we multiply the equation in \eqref{5.11} with $W_2^*$ from the left to get
\begin{align*}
\frac{d}{dt}(W_2^*W_{2})+\sum\limits_{j=1}^d \partial_{x_j}(W_2^*A_j
 W_2)\leq 0.
\end{align*}
Integrating the last inequality over $(x_1,\hat{x})\in[0,+\infty)\times \mathbb{R}^{d-1}$ and $t\in[0,T]$ yields
\begin{align*}
\max_{t\in[0,T]}\|W_2(t,\cdot,\cdot)\|_{L^2(\mathbb{R}^+\times \mathbb{R}^{d-1})}^2\leq & ~ C\int_{\mathbb{R}^{d-1}}\int_0^{T} |W_2(t,0,\hat{x})|^2dtd\hat{x}.
\end{align*} 
This together with \eqref{5.17} completes the proof.
\end{proof}
 
\begin{appendices}
\section{Appendix}\label{AppendA1}
Here we prove that the matrix $\left({\begin{array}{*{20}c}
\vspace{1mm}0 & K^*K \\
K^*K  & K^*AK 
\end{array}}\right)$ has exactly $n_1^o$ positive eigenvalues and $n_1^o$ negative eigenvalues.

Notice that $K^*K$ is symmetric positive definite.
Let $K_A=(K^*K)^{-1/2}(K^*AK)(K^*K)^{-1/2}$. Due to the congruent transformation
\begin{align*}
&\left({\begin{array}{*{20}c}
\vspace{1mm}(K^*K)^{-1/2} & 0 \\
0 &  (K^*K)^{-1/2}
\end{array}}\right)
\left({\begin{array}{*{20}c}
\vspace{1mm}0 & K^*K \\
K^*K  & K^*AK 
\end{array}}\right)
\left({\begin{array}{*{20}c}
\vspace{1mm}(K^*K)^{-1/2} & 0 \\
0 &  (K^*K)^{-1/2}
\end{array}}\right)\\[2mm]
=&\left({\begin{array}{*{20}c}
\vspace{1.5mm}0 & I_{n_1^o}\\
 I_{n_1^o} & K_A
\end{array}}\right),
\end{align*}
it suffices to show that the matrix $\left({\begin{array}{*{20}c}
\vspace{1.5mm}0 & I_{n_1^o}\\
 I_{n_1^o} & K_A
\end{array}}\right)$ has exactly $n_1^o$ positive eigenvalues and $n_1^o$ negative eigenvalues. Clearly, this matrix is invertible and its eigenvalue $\omega$ satisfies 
$$
|\omega^2I -\omega K_A-I |=0.
$$
Since $K_A$ is symmetric, there exists an orthonormal matrix $Q_A$ such that $Q_A^*Q_A=I_{n_1^o}$ and $Q_A^*K_AQ_A=\Lambda$ with $\Lambda$ an $n_1^o\times n_1^o$ diagonal matrix. Thus the last equation becomes 
$$
|\omega^2I -\omega \Lambda - I |=0.
$$
It is easy to see that this equation has exactly $n_1^o$ positive solutions and $n_1^o$ negative solutions. This is the proof.

\end{appendices}
 

\end{document}